\documentclass{amsart}
\usepackage{amsfonts,amssymb, wasysym}
\usepackage{multicol}

\newtheorem{theorem}{Theorem}
\newtheorem{lemma}{Lemma}

\newcommand{\integers}{{\mathbb Z}}
\newcommand{\realnos}{{\mathbb R}}

\def\ov{\overline}

\def\Beta{{\rm B}}
\def\Nu{{\rm N}}
\def\Mu{{\rm M}}
\def\Tau{{\rm T}}
\def\Eta{{\rm H}}

\def\Kappa{{\rm K}}

\def\Rho{{\rm P}}

\begin{document}

\title{A Bieberbach Theorem for Crystallographic Group Extensions}

\author{John G. Ratcliffe and Steven T. Tschantz}

\address{Department of Mathematics, Vanderbilt University, Nashville, TN 37240
\vspace{.1in}}

\email{j.g.ratcliffe@vanderbilt.edu}

\date{}

\begin{abstract}
In this paper we prove that for each dimension $n$ 
there are only finitely many isomorphism classes of pairs of groups $(\Gamma,\Nu)$ 
such that $\Gamma$ is an $n$-dimensional crystallographic group and 
$\Nu$ is a normal subgroup of $\Gamma$ such that $\Gamma/\Nu$ is a crystallographic group. 
\end{abstract}

\maketitle

\section{Introduction} 
An {\it $n$-dimensional crystallographic group} ({\it $n$-space group}) 
is a discrete group $\Gamma$ of isometries of Euclidean $n$-space $E^n$ 
whose orbit space $E^n/\Gamma$ is compact. 
The 3-space groups are the symmetry groups of crystalline structures,  
and so are of fundamental importance in the science of crystallography. 

In response to Hilbert's Problem 18, L. Bieberbach \cite{B} proved that for each dimension $n$ 
there are only finitely many isomorphism classes of $n$-space groups. 
In this paper we prove a relative version of Bieberbach's theorem. 
We prove that for each dimension $n$ 
there are only finitely many isomorphism classes of pairs of groups $(\Gamma,\Nu)$ 
such that $\Gamma$ is an $n$-space group and 
$\Nu$ is a normal subgroup of $\Gamma$ such that $\Gamma/\Nu$ is a space group. 

Our relative Bieberbach theorem has a geometric interpretation in the theory of flat orbifolds.  
By Theorems 7, 8, and 10 of \cite{R-T} the isomorphism classes of pairs of groups $(\Gamma, \Nu)$ 
such that $\Gamma$ is an $n$-space group and $\Nu$ is a normal subgroup of $\Gamma$ such that $\Gamma/\Nu$ is a space group  
correspond to the affine equivalence classes of geometric orbifold fibrations of compact,  
connected, flat $n$-orbifolds. 
Therefore, our relative Bieberbach theorem is equivalent to the theorem that for each dimension $n$ there are only finitely many 
affine equivalence classes of geometric orbifold fibrations of compact,  connected, flat $n$-orbifolds. 
This is known for $n = 3$ by the work of Conway-Friedrichs-Huson-Thurston\cite{C-T} and Ratcliffe-Tschantz \cite{R-T}. 

We now outline the proof of our relative Bieberbach theorem.
Let $m$ be a positive integer less than $n$. 
Let $\Mu$ be an $m$-space group and 
let $\Delta$ be an $(n-m)$-space group.  
Let $\mathrm{Iso}(\Delta,\Mu)$ be the set of isomorphism classes 
of pairs $(\Gamma, \Nu)$ 
where $\Nu$ is a normal subgroup of an $n$-space group $\Gamma$ 
such that $\Nu$ is isomorphic to $\Mu$ and $\Gamma/\Nu$ is isomorphic to $\Delta$. 
As there are only finitely many isomorphism classes of the groups $\Delta$ and $\Mu$ by 
Bieberbach's theorem \cite{B}, it suffices to prove that $\mathrm{Iso}(\Delta,\Mu)$ is finite. 

Next, we define a set $\mathrm{Out}(\Delta,\mathrm{M})$ in terms of $\mathrm{Out}(\Delta)$ and $\mathrm{Out}(\mathrm{M})$. 
That the set $\mathrm{Out}(\Delta,\mathrm{M})$ is finite follows easily from a theorem of Baues and Grunewald \cite{B-G} 
that the outer automorphism group of a crystallographic group is an arithmetic group. 
We define a function $\omega: \mathrm{Iso}(\Delta,\Mu) \to \mathrm{Out}(\Delta,\mathrm{M})$. 
We prove that $\mathrm{Iso}(\Delta,\Mu)$ is finite by showing that the fibers of $\omega$ are finite 
by a cohomology of groups argument. 

\section{Normal Subgroups of Space Groups}  

A map $\phi:E^n\to E^n$ is an isometry of $E^n$ 
if and only if there is an $a\in E^n$ and an $A\in {\rm O}(n)$ such that 
$\phi(x) = a + Ax$ for each $x$ in $E^n$. 
We shall write $\phi = a+ A$. 
In particular, every translation $\tau = a + I$ is an isometry of $E^n$. 

Let $\Gamma$ be an $n$-space group. 
Define $\eta:\Gamma \to {\rm O}(n)$ by $\eta(a+A) = A$. 
Then $\eta$ is a homomorphism whose kernel is the group $\mathrm{T}$ 
of translations in $\Gamma$. 
The image of $\eta$ is a finite group $\Pi$ called the 
{\it point group} of $\Gamma$.

Let $\Eta$ be a subgroup of an $n$-space group $\Gamma$. 
Define the {\it span} of $\Eta$ by the formula
$${\rm Span}(\Eta) = {\rm Span}\{a\in E^n:a+I\in \Eta\}.$$
Note that ${\rm Span}(\Eta)$ is a vector subspace $V$ of $E^n$.  
Let $V^\perp$ denote the orthogonal complement of $V$ in $E^n$. 

\begin{theorem} {\rm (Theorem 2 \cite{R-T})} 
Let ${\rm N}$ be a normal subgroup of an $n$-space group $\Gamma$, 
and let $V = {\rm Span}(\Nu)$. 
\begin{enumerate}
\item If $b+B\in\Gamma$, then $BV=V$. 
\item If $a+A\in \Nu$,  then $a\in V$ and $ V^\perp\subseteq{\rm Fix}(A)$. 
\item The group $\Nu$ acts effectively on each coset $V+x$ of $V$ in $E^n$ 
as a space group of isometries of $V+x$. 
\end{enumerate}
\end{theorem}

Let $\Gamma$ be an $n$-space group. 
The {\it dimension} of $\Gamma$ is $n$. 
If $\Nu$ is a normal subgroup of $\Gamma$, 
then $\Nu$ is a $m$-space group with $m= \mathrm{dim}(\mathrm{Span}(\Nu))$ 
by Theorem 1(3). 

\vspace{.15in}
\noindent{\bf Definition:}
Let $\Nu$ be a normal subgroup $\Nu$ of an $n$-space group $\Gamma$, and let $V = {\rm Span}(\Nu)$.  Then $\Nu$ is said to be a {\it complete normal subgroup} of $\Gamma$ if 
$$\Nu= \{a+A\in \Gamma: a\in V\ \hbox{and}\ V^\perp\subseteq{\rm Fix}(A)\}.$$

\begin{lemma}  {\rm (Lemma 1 \cite{R-T})} 
Let $\Nu$ be a complete normal subgroup of an $n$-space group $\Gamma$, 
and let $V={\rm Span}(\Nu)$. 
Then $\Gamma/\Nu$ acts effectively as a space group of isometries of $E^n/V$ 
by the formula
$({\rm N}(b+B))(V+x) = V+ b+Bx.$
\end{lemma}

\noindent{\bf Remark 1.} A normal subgroup $\Nu$ of a space group $\Gamma$ is complete 
precisely when $\Gamma/\Nu$ is a space group by Theorem 5 of \cite{R-T}.

\vspace{.15in}

Let $\Nu$ be a complete normal subgroup of an $n$-space group $\Gamma$, 
let $V = \mathrm{Span}(\Nu)$, and let $V^\perp$ be the orthogonal complement of $V$ in $E^n$.  
Let $\gamma \in \Gamma$.   Then $\gamma = b+B$ with $b\in E^n$ and $B\in \mathrm{O}(n)$. 
Write $b = \overline b + b'$ with $\overline b \in V$ and $b' \in V^\perp$. 
Let $\overline B$ and $B'$ be the orthogonal transformations of $V$ and $V^\perp$, respectively, 
obtained by restricting $B$. 
Let $\overline \gamma = \overline b + \overline B$ and $\gamma' = b' + B'$. 
Then $\overline \gamma$ and $\gamma'$ are isometries of $V$ and $V^\perp$, respectively.

Euclidean $n$-space $E^n$ decomposes as the Cartesian product $E^n = V \times V^\perp$. 
Let $x\in E^n$.  Write $x = v+w$ with $v\in V$ and $w\in V^\perp$. Then 
$$(b+B)x = b+Bx = \overline{b}+b' + Bv + Bw = (\overline{b}+\overline{B}v) + (b'+B'w).$$
Hence the action of $\Gamma$ on $E^n$ corresponds to the diagonal action of $\Gamma$ 
on $V\times V^\perp$ defined by the formula
$$\gamma(v,w) = (\overline{\gamma}v,\gamma'w).$$
Here $\Gamma$ acts on both $V$ and $V^\perp$ via isometries. 
The kernel of the corresponding homomorphism from $\Gamma$ to $\mathrm{Isom}(V)$ 
is the group
$$\Kappa = \{b+B\in\Gamma: b \in V^\perp\ \hbox{and}\ V \subseteq \mathrm{Fix}(B)\}.$$
We call $\Kappa$ the {\it kernel of the action} of $\Gamma$ on $V$. 
The group $\Kappa$ is a normal subgroup of $\Gamma$. 
The action of $\Gamma$ on $V$ induces an effective action of $\Gamma/\Kappa$ on $V$ via isometries.  
The group $\Gamma/\Kappa$ acts on $V$ as a discrete group of isometries if and only if 
$\Gamma/\Nu\Kappa$ is a finite group by Theorem 3(4) of \cite{R-T-Cal}. 

The group $\Nu$ is the kernel of the action of $\Gamma$ on $V^\perp$, and so the action 
of $\Gamma$ on $V^\perp$ induces  an effective action of $\Gamma/\Nu$ on $V^\perp$ via isometries. 
Orthogonal projection from $E^n$ to $V^\perp$ induces an isometry from $E^n/V$ to $V^\perp$. 
Hence $\Gamma/\Nu$ acts on $V^\perp$ as a space group of isometries by Lemma 1. 

Let $\ov \Gamma = \{\ov \gamma: \gamma \in \Gamma\}$. 
If $\gamma \in \Gamma$, then $(\overline{\gamma})^{-1} = \overline{\gamma^{-1}}$, and 
if $\gamma_1, \gamma_2 \in \Gamma$, 
then $\overline{\gamma_1}\,\overline{\gamma_2}= \overline{\gamma_1\gamma_2}$.   
Hence $\ov \Gamma$ is a subgroup of $\mathrm{Isom}(V)$. 
The map $\Beta: \Gamma \to \ov \Gamma$ defined by $\Beta(\gamma)= \ov\gamma$ is an epimorphism 
with kernel $\Kappa$. 
The group $\ov\Gamma$ is a discrete subgroup of $\mathrm{Isom}(V)$ 
if and only if $\Gamma/\Nu\Kappa$ is finite by Theorem 3(4) of \cite{R-T-Cal}. 

Let $\Gamma' = \{\gamma' : \gamma \in \Gamma\}$.  
If $\gamma \in \Gamma$, then $(\gamma')^{-1} = (\gamma^{-1})'$, and 
if $\gamma_1, \gamma_2 \in \Gamma$, 
then $\gamma_1'\gamma_2' = (\gamma_1\gamma_2)'$.   
Hence $\Gamma'$ is a subgroup of $\mathrm{Isom}(V^\perp)$. 
The map $\Rho' : \Gamma \to \Gamma'$ 
defined by $\Rho'(\gamma) = \gamma'$ is epimorphism with kernel $\Nu$, 
since $\Nu$ is a complete normal subgroup of $\Gamma$.  
Hence $\Rho'$ induces an isomorphism $\Rho: \Gamma/\Nu \to \Gamma'$ defined by $\Rho(\Nu\gamma) = \gamma'$. 
The group $\Gamma'$ is a space group of isometries of $V^\perp$ 
with $V^\perp/\Gamma' = V^\perp/(\Gamma/\Nu)$. 

Let $\ov \Nu = \{\ov \nu: \nu \in \Nu\}$.  
Then $\ov \Nu$ is a subgroup of $\mathrm{Isom}(V)$. 
The map $\Beta: \Nu \to \ov \Nu$ defined by $\Beta(\nu) = \ov\nu$ is an isomorphism. 
The group $\ov\Nu$ is a space group of isometries of $V$ with $V/\ov{\Nu} = V/\Nu$. 

The action of $\Gamma$ on $V$ induces an action of $\Gamma/\Nu$ on $V/\Nu$ 
defined by 
$$(\Nu\gamma)(\Nu v) = \Nu \overline{\gamma}v.$$ 
The action of $\Gamma/\Nu$ on $V/\Nu$ determines a homomorphism 
$$\Xi : \Gamma/\Nu \to \mathrm{Isom}(V/\Nu)$$
defined by $\Xi(\Nu \gamma) = \overline \gamma_\star$, where $\overline \gamma_\star: V/\Nu \to V/\Nu$ 
is defined by $\overline\gamma_\star(\Nu v) = \Nu \overline\gamma(v)$.

\begin{theorem} 
Let $\Mu$ be an $m$-space group, let $\Delta$ be an $(n-m)$-space group, 
and let $\Theta:\Delta \to \mathrm{Isom}(E^m/\Mu)$ be a homomorphism. 
Identify $E^n$ with $E^m\times E^{n-m}$, and extend $\Mu$ to a subgroup $\Nu$ of $\mathrm{Isom}(E^n)$
such that the point group of $\Nu$ acts trivially on $(E^m)^\perp= E^{n-m}$. 
Then there exists a unique $n$-space group $\Gamma$ containing $\Nu$ as a complete normal subgroup 
such that $\Gamma' = \Delta$, and if $\Xi:\Gamma/\Nu \to \mathrm{Isom}(E^m/\Nu)$ is 
the homomorphism induced by the action of $\Gamma/\Nu$ on $E^m/\Nu$, then $\Xi = \Theta\Rho$ 
where $\Rho:\Gamma/\Nu \to \Gamma'$ is the isomorphism defined by $\Rho(\Nu\gamma) = \gamma'$ 
for each $\gamma\in\Gamma$. 
\end{theorem}
\begin{proof}  Let $\delta \in \Delta$.  
By Lemma 1 of \cite{R-T-Isom}, there exists $\tilde\delta \in N_E(\Mu)$ such that $\tilde\delta_\star = \Theta(\delta)$. 
The isometry $\tilde\delta$ is unique up to multiplication by an element of $\Mu$. 
Let $\delta = d + D$, with $d\in E^{n-m}$ and $D\in \mathrm{O}(n-m)$, 
and let $\tilde\delta = \tilde d + \tilde D$, with $\tilde d \in E^m$ and $\tilde D \in \mathrm{O}(m)$.  
Let $\hat d = \tilde d + d$, and let $\hat D = \tilde D \times D$. 
Then $\hat D E^m = E^m$ and $\hat D E^{n-m} = E^{n-m}$. 
Let $\hat \delta = \hat d + \hat D$. 
Then $\hat \delta$ is an isometry of $E^n$ such that $\overline{\hat\delta} = \tilde\delta$ 
and $\big(\hat\delta\big)' = \delta$.  
The isometry $\hat \delta$ is unique up to multiplication by an element of $\Nu$. 
We have that 
$$\hat\delta^{-1} = -\hat D^{-1}\hat d + \hat D^{-1} = -\tilde D^{-1}\tilde d -D^{-1}d + \tilde D^{-1} \times D^{-1},$$
and so $\overline{\hat\delta^{-1} }= \big(\tilde \delta\big)^{-1}$ and  $(\hat\delta^{-1})' = \delta^{-1}$. 
We have that 
$$\ov{\hat\delta\Nu\hat\delta^{-1}} = \ov{\hat\delta}\ov\Nu\ov{\hat\delta^{-1}} = \tilde\delta\Mu\tilde\delta^{-1} = \Mu = \ov\Nu$$
and 
$$(\hat\delta\Nu\hat\delta^{-1})' = \hat\delta'\Nu'(\hat\delta^{-1})' = \delta\{I'\}\delta^{-1} = \{I'\}.$$
Therefore $\hat \delta\Nu\hat\delta^{-1}  = \Nu$. 

Let  $\Gamma$ be the subgroup of $\mathrm{Isom}(E^n)$ 
generated by $\Nu \cup \{\hat \delta : \delta \in \Delta\}$.  
Then $\Gamma$ contains $\Nu$ as a normal subgroup, 
and the point group of $\Gamma$ leaves $E^m$ invariant. 
Suppose $\gamma \in \Gamma$.  Then there exists $\nu \in \Nu$ and $\delta_1, \ldots, \delta_k \in \Delta$ 
and $\epsilon_1,\ldots, \epsilon_k \in \{\pm 1\}$ 
such that $\gamma = \nu \hat{\delta}_1^{\epsilon_1} \cdots \hat{\delta}_k^{\epsilon_k}$. 
Then we have 
$$\gamma' =  \nu' (\hat{\delta}_1^{\epsilon_1})' \cdots (\hat{\delta}_k^{\epsilon_k})' = \delta_1^{\epsilon_1}\cdots \delta_k^{\epsilon_k}.$$
Hence $\Gamma' = \Delta$, and 
we have an epimorphism $\Rho':\Gamma \to \Delta$ defined by $\Rho'(\gamma) = \gamma'$. 
The group $\Nu$ is in the kernel of $\Rho'$, and so $\Rho'$ induces an epimorphism $\Rho: \Gamma/\Nu \to \Delta$ 
defined by $\Rho(\Nu\gamma) = \gamma'$. 
Suppose $\Rho(\Nu\gamma) = I'$.  By Lemma 1 of \cite{R-T-Isom}, we have that 
\begin{eqnarray*}
 \Rho(\Nu\gamma) = I' \ \ \Rightarrow \ \  \gamma' = I' & \Rightarrow  & \delta_1^{\epsilon_1}\cdots \delta_k^{\epsilon_k} =I' \\
       			     & \Rightarrow & \Theta(\delta_1)^{\epsilon_1} \cdots \Theta(\delta_k)^{\epsilon_k} = \ov I_\star\\
			    & \Rightarrow &(\tilde \delta_1)_\star^{\epsilon_1} \cdots (\tilde \delta_k)_\star^{\epsilon_k} = \ov I_\star\\
		             & \Rightarrow &\Mu(\tilde \delta_1)^{\epsilon_1} \cdots \Mu(\tilde \delta_k)^{\epsilon_k} = \Mu\\
			    & \Rightarrow &\Mu (\tilde \delta_1)^{\epsilon_1} \cdots (\tilde \delta_k)^{\epsilon_k} = \Mu \\
		             & \Rightarrow &\Mu \overline{\hat \delta_1^{\epsilon_1}} \cdots \overline{\hat \delta_k^{\epsilon_k}} = \Mu
			   \ \  \Rightarrow \ \ \Mu \overline\gamma = \Mu  \ \  \Rightarrow \ \ \overline\gamma \in \Mu. 
\end{eqnarray*}
As $\gamma' = I'$ and $\overline \gamma \in \ov\Nu$, we have that $\gamma \in \Nu$. 
Thus $\Rho$ is an isomorphism. 

We next show that $\Gamma$ acts discontinuously on $E^n$. 
Let $C$ be a compact subset of $E^n$. 
Let $K$ and $L$ be the orthogonal projections of $C$ into $E^m$ and $E^{n-m}$, respectively. 
Then $C \subseteq K \times L$.  
As $\Delta$ acts discontinuously on $E^{n-m}$, there exists only finitely many elements $\delta_1, \ldots, \delta_k$ of $\Delta$  
such that $L\cap \delta_iL \neq \emptyset$ for each $i$. 
Let $\gamma_i = \hat \delta_i$ for $i = 1,\ldots, k$. 
The set $K_i = K \cup \overline \gamma_i(K)$ is compact for each $i = 1,\ldots, k$. 
As $\Nu$ acts discontinuously on $E^m$, there is a finite subset $F_i$ of $\Nu$ such that $K_i \cap \nu K_i \neq \emptyset$ 
only if $\nu \in F_i$ for each $i$.  Hence $K \cap \nu\overline\gamma_i(K)\neq\emptyset$ only if $\nu \in F_i$ for each $i$. 
The set $F = F_1\gamma_1\cup \cdots \cup F_k\gamma_k$ is a finite subset of $\Gamma$.  
Suppose $\gamma \in \Gamma$ and  $C\cap \gamma C \neq \emptyset$. 
Then $L \cap \gamma' L \neq \emptyset$, and so $ \gamma' = \delta_i$ for some $i$. 
Hence $\gamma = \nu \gamma_i$ for some $\nu \in \Nu$. 
Now we have that $K\cap \overline \gamma K \neq \emptyset$, and so $K\cap \nu\overline\gamma_iK \neq \emptyset$. 
Hence $\nu \in F_i$.  Therefore $\gamma \in F$.  Thus $\Gamma$ acts discontinuously on $E^n$. 
Therefore $\Gamma$ is a discrete subgroup of $\mathrm{Isom}(E^n)$ by Theorem 5.3.5 of \cite{R}. 

Let $D_\Mu$ and $D_\Delta$ be fundamental domains for $\Mu$ in $E^m$ and $\Delta$ in $E^{n-m}$, respectively. 
Then their topological closures $\overline{D}_\Mu$ and $\overline{D}_\Delta$ are compact sets. 
Let $x \in E^n$.  Write $x = \overline x + x'$ with $\overline x \in E^m$ and $x' \in E^{n-m}$. 
Then there exists $\delta \in \Delta$ such that $\delta x' \in \overline D_\Delta$, 
and there exists $\nu \in \Nu$ such that $\nu \tilde\delta\overline x \in \overline D_\Mu$. 
We have that 
$$\nu\hat \delta x = \overline{ \nu\hat\delta} \overline x +  (\nu \hat\delta)'x' = \nu\tilde \delta \overline x + \delta x' 
\in \overline D_\Mu \times \overline D_\Delta.$$
Hence the quotient map $\pi: E^n \to E^n/\Gamma$ maps the compact set $\overline D_\Mu \times \overline D_\Delta$ 
onto $E^n/\Gamma$.  Therefore $E^n/\Gamma$ is compact.  Thus $\Gamma$ is an $n$-space group. 

Let $\Xi: \Gamma/\Nu \to \mathrm{Isom}(E^n/\Nu)$ be the homomorphism induced by the action of $\Gamma/\Nu$ on $E^m/\Nu$. 
Let $\gamma \in \Gamma$.  Then there exists $\delta \in \Delta$ such that $\Nu\gamma = \Nu\hat \delta$. 
Then $\gamma' = (\hat\delta)' = \delta$, and we have that $\Xi = \Theta\Rho$, since 
$$\Xi(\Nu\gamma) = \Xi(\Nu\hat\delta) = (\overline{\hat\delta})_\star = \tilde \delta_\star = 
\Theta(\delta) = \Theta(\gamma') = \Theta\Rho(\Nu\gamma).$$

Suppose $\gamma$ is an isometry of $E^n$ such that $\gamma\Nu\gamma^{-1} = \Nu$ 
and $\gamma' \in \Delta$ and $\overline\gamma_\star = \Theta(\gamma')$. 
Then $\widehat{\gamma'} \in \Gamma$. 
Now $\overline\gamma_\star = \widetilde{\gamma'}_\star$, 
and so $\overline\gamma = \ov\nu\widetilde{\gamma'}$ for some $\nu\in \Nu$ by Lemma 1 of \cite{R-T-Isom}. 
Then $\gamma = \nu\widehat{\gamma'}$.  Hence $\gamma\in \Gamma$. 
Thus $\Gamma$ is the unique $n$-space group that contains $\Nu$ as a complete normal subgroup such that $\Gamma' = \Delta$ 
and $\Xi = \Theta\Rho$. 
\end{proof}

\section{Isomorphisms of Pairs of Space Groups}  

An {\it affinity} $\alpha$ of $E^n$ is a map $\alpha: E^n\to E^n$ 
for which there is an element $a\in E^n$ and a matrix $A \in \mathrm{GL}(n,\realnos)$ such that 
$\alpha(x) = a + Ax$ for all $x$ in $E^n$.  We write simply $\alpha = a + A$. 
The set $\mathrm{Aff}(E^n)$ of all affinities of $E^n$ is a group 
that contains $\mathrm{Isom}(E^n)$ as a subgroup.

Let $\Nu_i$ be a complete normal subgroup of an $n$-space group $\Gamma_i$ for $i=1,2$. 
We want to know when $(\Gamma_1,\Nu_1)$ is isomorphic to $(\Gamma_2, \Nu_2)$, 
that is, when there is an isomorphism $\zeta: \Gamma_1 \to \Gamma_2$ such that $\zeta(\Nu_1) = \Nu_2$. 
By a Theorem of Bieberbach an isomorphism $\zeta: \Gamma_1 \to \Gamma_2$ is equal to 
conjugation by an affinity $\phi$ of $E^n$. 
In this section, we determine necessary and sufficient conditions such that there exists 
an affinity $\phi$ of $E^n$ such that $\phi(\Gamma_1,\Nu_1)\phi^{-1}=(\Gamma_2,\Nu_2)$. 

Let $\phi$ be an affinity of $E^n$ such that $\phi(\Gamma_1,\Nu_1)\phi^{-1}=(\Gamma_2,\Nu_2)$. 
Write $\phi = c+C$ with $c\in E^n$ and $C\in \mathrm{GL}(n,\realnos)$. 
Let $V_i = \mathrm{Span}(\Nu_i)$, for $i=1,2$. 
Let $a+I\in \Nu_1$.  Then $\phi(a+I)\phi^{-1} = Ca+I$.  Hence $CV_1 \subseteq V_2$. 
Let $\overline{C}:V_1 \to V_2$ be the linear transformation obtained by restricting $C$. 
Let $\overline{C'}:V_1^\perp \to V_2$ and $C': V_1^\perp \to V_2^\perp$ be the linear transformations 
obtained by restricting $C$ to $V_1^\perp$ followed by the orthogonal projections to $V_2$ and $V_2^\perp$, respectively. 
Write $c = \overline{c} + c'$ with $\ov{c}\in V_2$ and $c'\in V_2^\perp$. 
Let $\overline{\phi}: V_1 \to V_2$ and $\phi': V_1^\perp \to V_2^\perp$ be the affine transformations 
defined by $\overline{\phi} = \overline{c}+\ov{C}$ and $\phi' = c'+C'$. 

\begin{lemma} 
Let $\Nu_i$ be a complete normal subgroup of an $n$-space group $\Gamma_i$, with $V_i = \mathrm{Span}(\Nu_i)$, for $i=1,2$. 
Let $\phi = c+C$ be an affinity of $E^n$ such that $\phi(\Gamma_1,\Nu_1)\phi^{-1}=(\Gamma_2,\Nu_2)$. 
Then $\ov{C}$ and $C'$ are invertible, with $\ov{C}^{-1} = \ov{C^{-1}}$ and $(C')^{-1} = (C^{-1})'$,  
and $\ov{(C^{-1})'} = -\ov{C}^{-1}\ov{C'}(C')^{-1}$.  
If $b+B\in \Gamma_1$, then $\ov{C'}B' = \ov C \ov B \ov C^{-1} \ov{C'}$.  
Moreover $\ov{C'}V_1^\perp \subseteq \mathrm{Span}(Z(\Nu_2))$. 
\end{lemma}
\begin{proof} We have that 
$$\dim V_1 = \dim \Nu_1 = \dim \Nu_2 = \dim V_2.$$
Let $y\in E^n$ and write $y = \ov{y} + y'$ with $\ov{y}\in V_2$ and $y'\in V_2^\perp$. 
Then 
$$\ov{y} = CC^{-1}(\ov{y}) = \ov{C}\ov{C^{-1}}(\ov{y}),$$
and so $\ov{C}$ is invertible with $\ov{C}^{-1} = \ov{C^{-1}}$. 
We have that 
\begin{eqnarray*}
y' & = & CC^{-1}y' \\
& = &C(\ov{(C^{-1})'}(y') + (C^{-1})'(y')) \\
& = &C\ov{(C^{-1})'}(y')+ C(C^{-1})'(y') \\
& = & \ov{C}\ov{(C^{-1})'}(y')+\ov{C'}(C^{-1})'(y')+ C'(C^{-1})'(y'). 
\end{eqnarray*}
Hence $C'$ is invertible, with $(C')^{-1} = (C^{-1})'$, and $\ov{C}\ov{(C^{-1})'}+\ov{C'}(C^{-1})' =0$. 
Therefore $\ov{(C^{-1})'} = -\ov{C}^{-1}\ov{C'}(C')^{-1}$. 

Let $\gamma = b+ B \in \Gamma_1$.  Then $\phi\gamma\phi^{-1} \in \Gamma_2$.  Let $w \in V_2^\perp$. 
Then we have 
\begin{eqnarray*}
w& = & CBC^{-1}w\\
& = &CB(\ov{(C^{-1})'}(w) + (C^{-1})'(w)) \\
& = &C(\ov B\ov{(C^{-1})'}(w)+ B'(C^{-1})'(w)) \\
& = &C\ov B\ov{(C^{-1})'}(w)+ CB'(C^{-1})'(w) \\
& = & \ov{C}\ov B\ov{(C^{-1})'}(w)+\ov{C'}B'(C^{-1})'(w)+ C'B'(C^{-1})'(w). \\
\end{eqnarray*}
As $\phi\gamma\phi^{-1}(w) \in V_2^\perp$, 
we have that 
$$\ov{C}\ov B\ov{(C^{-1})'}+\ov{C'}B'(C^{-1})'=0.$$
Therefore 
$$\ov{C}\ov B(-\ov{C}^{-1}\ov{C'}(C')^{-1})+\ov{C'}B'(C')^{-1}=0.$$
Hence $\ov{C'}B' = \ov C \ov B\ov{C}^{-1}\ov{C'}$. 

Now suppose $\gamma \in \Nu_1$.  Then $B' = I'$ by Theorem 1(2). 
Hence $\ov B\ov{C}^{-1}\ov{C'} = \ov{C}^{-1}\ov{C'}$. 
By Lemma 5 of \cite{R-T-Isom}, we deduce that $\ov{C}^{-1}\ov{C'}V_1^\perp \subseteq \mathrm{Span}(Z(\Nu_1))$. 
As $\phi Z(\Nu_1)\phi^{-1} = Z(\Nu_2)$, we have  
that $\ov{C'}V_1^\perp \subseteq \mathrm{Span}(Z(\Nu_2))$. 
\end{proof}

\begin{theorem}  
Let $\Nu_i$ be  a complete normal subgroup of an $n$-space group $\Gamma_i$, with $V_i = \mathrm{Span}(\Nu_i)$ for $i=1,2$. 
Let $\Xi_i:\Gamma_i/\Nu_i \to \mathrm{Isom}(V_i/\Nu_i)$ be the homomorphism induced 
by the action of $\Gamma_i/\Nu_i$ on $V_i/\Nu_i$ for $i=1,2$.  
Let $\alpha: V_1 \to V_2$, and $\beta: V_1^\perp \to V_2^\perp$ be affinities such that $\alpha\ov\Nu_1\alpha^{-1} = \ov\Nu_2$ 
and $\beta\Gamma'_1\beta^{-1} = \Gamma_2'$. 
Let $\alpha_\star: V_1/\Nu_1 \to V_2/\Nu_2$ be the affinity defined by $\alpha_\star(\Nu_1v) = \Nu_2\alpha(v)$, 
and let $\alpha_\sharp: \mathrm{Aff}(V_1/\Nu_1) \to \mathrm{Aff}(V_2/\Nu_2)$ be the isomorphism 
defined by $\alpha_\sharp(\chi) = \alpha_\star\chi\alpha_\star^{-1}$.  
Let $\beta_\ast: \Gamma_1' \to \Gamma_2'$ be the isomorphism defined by $\beta_\ast(\gamma') = \beta\gamma'\beta^{-1}$. 

Write $\alpha = \ov c + \ov C$ with $\ov c \in V_2$ and $\ov C: V_1 \to V_2$ a linear isomorphism. 
Let $D: V_1^\perp \to \mathrm{Span}(Z(\Nu_2))$ be a linear transformation 
such that if $b+B\in \Gamma_1$, then $DB' = \ov C\ov B\ov C^{-1} D$. 
Let $\Rho_i:\Gamma_i/\Nu_i\to \Gamma_i'$ be the isomorphism 
defined by $\Rho_i(\Nu_i\gamma) = \gamma'$ for each $i = 1,2$, 
and let $p_1:\Gamma_1/\Nu_1 \to V_1^\perp$ be defined by $p_1(\Nu_1(b+B)) = b'$. 
Let $\mathcal{K}_i$ be the connected component of the identity of the Lie group $\mathrm{Isom}(V_i/\Nu_i)$ for each $i = 1,2$. 
Note that $\mathcal{K}_i = \{(v+\ov I)_\star: v \in \mathrm{Span}(Z(\Nu_i))\}$ by Theorem 1 of \cite{R-T-Isom}. 
Then the following are equivalent:

\begin{enumerate}
\item There exists an affinity $\phi = c + C$ of $E^n$ such that $\phi(\Gamma_1,\Nu_1)\phi^{-1} = (\Gamma_2,\Nu_2)$, 
with $\ov\phi = \alpha$, and $\phi' = \beta$, and $\ov{C'} = D$. 
\item We have that  
$$\Xi_2\Rho_2^{-1}\beta_\ast\Rho_1 = (Dp_1)_\star\alpha_\sharp\Xi_1$$
with $(Dp_1)_\star:\Gamma_1/\Nu_1 \to \mathcal{K}_2$ a crossed homomorphism defined by
$$(Dp_1)_\star(\Nu_1(b+B)) = (Db'+\ov I)_\star$$
and $\Gamma_1/\Nu_1$ acting on $\mathcal{K}_2$ by $\Nu_1(b+B)(v+\ov I)_\star =  (\ov C\ov B\ov C^{-1}v+\ov I)_\star$ 
for each $b+B \in \Gamma_1$ and $v \in  \mathrm{Span}(Z(\Nu_2))$. 
\item We have that
$$\alpha_\sharp^{-1}\Xi_2\Rho_2^{-1} \beta_\ast\Rho_1 = (\ov C^{-1}Dp_1)_\star\Xi_1$$
with $(\ov C^{-1}Dp_1)_\star: \Gamma_1/\Nu_1 \to \mathcal{K}_1$ a crossed homomorphism 
defined by 
$$(\ov C^{-1}Dp_1)_\star(\Nu_1(b+B)) = (\ov C^{-1}D(b')+\ov I)_\star$$
and $\Gamma_1/\Nu_1$ acting on $\mathcal{K}_1$ by $\Nu_1(b+B)(v + \ov I)_\star = (\ov B v+\ov I)_\star$ for each $b+B \in \Gamma_1$ 
and $v \in \mathrm{Span}(Z(\Nu_1))$. 
\end{enumerate}
\end{theorem}
\begin{proof}
Suppose there exists an affinity $\phi = c+ C$ of $E^n$ such that $\phi(\Gamma_1,\Nu_1)\phi^{-1} = (\Gamma_2,\Nu_2)$, 
with $\ov\phi = \alpha$, and $\phi' = \beta$, and $\ov{C'} = D$. 
Let $\gamma  = b+ B \in \Gamma_1$.  
Then we have 
$$\phi\gamma\phi^{-1} = (c+C)(b+B)(c+C)^{-1} = Cb + (I-CBC^{-1})c+CBC^{-1}.$$
Hence we have 
\begin{eqnarray*}
\ov{\phi\gamma\phi^{-1}} & = & \ov{C}\ov b + \ov{C'}b'+ (\ov I-\ov C\ov B{\ov C}^{-1})\ov c+\ov C\ov B{\ov C}^{-1} \\
& = & ( \ov{C'}b'+\ov I)(\ov c +{\ov C})(\ov b+\ov B)(\ov c + {\ov C})^{-1} \\
& = & (Db'+\ov I)\ov \phi\ov \gamma\ov \phi^{-1} \ \
 = \ \ ( Db'+\ov I)\alpha\ov \gamma\alpha^{-1},  
\end{eqnarray*}
and 
\begin{eqnarray*}
(\phi\gamma\phi^{-1})' & = & C'b'+ (I'-C'B'(C')^{-1})c' + C'B'(C')^{-1} \\
& = & (c'+C')(b'+B')(c'+C')^{-1} \\ 
& = & \phi' \gamma' (\phi')^{-1} \ \ = \ \ \beta \gamma' \beta^{-1}.
\end{eqnarray*}
Observe that 

\begin{eqnarray*} \Xi_2\Rho_2^{-1}\beta_\ast\Rho_1(\Nu_1\gamma) 
& = & \Xi_2\Rho_2^{-1}\beta_\ast(\gamma') \\
& = & \Xi_2\Rho_2^{-1}(\beta\gamma'\beta^{-1}) \\
& = & \Xi_2\Rho_2^{-1}(\phi\gamma\phi^{-1})' \\
& = & \Xi_2(\Nu_2\phi\gamma\phi^{-1})  \\
& =  & (\ov{\phi\gamma\phi^{-1}})_\star \\
& =  & ((Db'+\ov I)\alpha\ov \gamma\alpha^{-1})_\star \\
& =  & (Db'+\ov I)_\star\alpha_\star\ov \gamma_\star \alpha_\star^{-1}  \\
& =  & (Dp_1(\Nu_1(b+B)))_\star \alpha_\sharp(\ov \gamma_\star) \ \,
 =  \ \, (Dp_1(\Nu_1\gamma))_\star \alpha_\sharp(\Xi_1(\Nu_1\gamma)). 
\end{eqnarray*}
Hence we have that $\Xi_2\Rho_2^{-1}\beta_\ast\Rho_1 = (Dp_1)_\star\alpha_\sharp\Xi_1$. 

Conversely, suppose that $\Xi_2\Rho_2^{-1}\beta_\ast\Rho_1 = (Dp_1)_\star\alpha_\sharp\Xi_1$. 
Define an affine transformation $\phi: E^n \to E^n$ by 
$$\phi(x) = \alpha(\overline x) + D(x')+ \beta(x')$$
for each $x\in E^n$, where $x = \ov x + x'$ with $\ov x \in V_1$ and $x'\in V_1^\perp$. 
Then $\phi$ is an affinity of $E^n$, with
$$\phi^{-1}(y) = \alpha^{-1}(\ov y) -\ov C^{-1}D\beta^{-1}(y') + \beta^{-1}(y')$$
for each $y\in E^n$ where $y = \ov y + y'$ with $\ov y \in V_2$  and $y'\in V_2^\perp$. 

Write $\beta = c' + C'$ with $c' \in V_2^\perp$ and $C': V_1^\perp \to V_2^\perp$ a linear isomorphism. 
Write $\phi = c+ C$ with $c \in E^n$ and $C$ a linear isomorphism of $E^n$. 
Then $c =\ov c + c'$ and $Cx = \ov C\ov x + Dx' + C'x'$ for each $x\in E^n$ with $\ov x \in V_1$ and $x'\in V_1^\perp$. 
We have that $\ov \phi = \alpha$, and $\phi' = \beta$ and $\ov{C'} = D$. 
We also have $\phi^{-1} = -C^{-1}c+C^{-1}$ and 
$$C^{-1}y = \ov C^{-1}\ov y - \ov C^{-1}D(C')^{-1}y' + (C')^{-1}y'$$
for all $y \in E^n$ with $\ov y \in V_2$ and $y' \in V_2^\perp$. 

Let $\gamma \in \Gamma_1$.  Write $\gamma = b+ B$ with $b\in E^n$ and $B\in \mathrm{O}(n)$. 
Then we have 
$$ \phi\gamma\phi^{-1} = (c+C)(b+B)(c+C)^{-1} = Cb+(I-CBC^{-1})c + CBC^{-1}.$$
Let $y \in E^n$.  Write $y = \ov y + y'$ with $\ov y \in V_2$ and $y' \in V_2^\perp$. 
Then we have that 
\begin{eqnarray*} CBC^{-1} 
& = & CB(\ov C^{-1}\ov y - \ov C^{-1}D(C')^{-1}y' + (C')^{-1}y'  \\
& = & C(\ov B\ov C^{-1}\ov y -\ov B\ov C^{-1}D(C')^{-1}y' +B'(C')^{-1}y' \\
& = & \ov C\ov B\ov C^{-1}\ov y -\ov C\ov B\ov C^{-1}D(C')^{-1}y' +DB'(C')^{-1}y' + C'B'(C')^{-1}y' \\
& = & \ov C\ov B\ov C^{-1}\ov y + C'B'(C')^{-1}y'.
\end{eqnarray*}
Hence $CBC^{-1} =  \ov C\ov B\ov C^{-1} \times C'B'(C')^{-1}$ as a linear isomorphism 
of $E^n = V_2\times V_2^\perp$. 

Moreover, we have 
\begin{eqnarray*}
\ov{\phi\gamma\phi^{-1}} & = & \ov{C}\ov b + \ov{C'}b'+ (\ov I-\ov C\ov B{\ov C}^{-1})\ov c+\ov C\ov B{\ov C}^{-1} \\
& = & ( \ov{C'}b'+\ov I)(\ov c +{\ov C})(\ov b+\ov B)(\ov c + {\ov C})^{-1} \\
& = & (Db'+\ov I)\ov \phi\ov \gamma\ov \phi^{-1} \ \
 = \ \ ( Db'+\ov I)\alpha\ov \gamma\alpha^{-1},  
\end{eqnarray*}
and 
\begin{eqnarray*}
(\phi\gamma\phi^{-1})' & = & C'b'+ (I'-C'B'(C')^{-1})c' + C'B'(C')^{-1} \\
& = & (c'+C')(b'+B')(c'+C')^{-1} \\ 
& = & \phi' \gamma' (\phi')^{-1} \ \ = \ \ \beta \gamma' \beta^{-1}.
\end{eqnarray*}

As $\Xi_2\Rho_2^{-1}\beta_\ast\Rho_1 = (Dp_1)_\star\alpha_\sharp\Xi_1$, 
we have that $(\alpha\ov\gamma\alpha^{-1})_\star$ is an isometry of $V_2/\Nu_2$. 
By Lemmas 1 and 7 of \cite{R-T-Isom}, we have that $\alpha\ov\gamma\alpha^{-1}$ is an isometry of $V_2$. 
Hence $\ov C \ov B \ov C^{-1}$ is an orthogonal transformation of $V_2$. 

As $\beta\Gamma_1'\beta^{-1} = \Gamma_2'$, we have that $C'B'(C')^{-1}$  is an orthogonal transformation of $V_2^\perp$.
Hence 
$CBC^{-1} = \ov C\ov B\ov C^{-1} \times C'B'(C')^{-1}$
is an orthogonal transformation of $E^n = V_2 \times V_2^\perp$.
Therefore $\phi\gamma\phi^{-1}$ is an isometry of $E^n$ for each $\gamma \in \Gamma_1$. 

As $\Gamma_1$ acts discontinuously on $E^n$ and $\phi$ is a homeomorphism of $E^n$, 
we have that $\phi\Gamma_1\phi^{-1}$ acts discontinuously on $E^n$. 
Therefore $\phi\Gamma_1\phi^{-1}$ is a discrete subgroup of $\mathrm{Isom}(E^n)$ by Theorem 5.3.5 of \cite{R}. 
Now $\phi$ induces a homeomorphism 
$\phi_\star: E^n/\Gamma_1 \to E^n/\phi\Gamma_1\phi^{-1}$
defined by $\phi_\star(\Gamma_1 x) = \phi\Gamma_1\phi^{-1}\phi(x)$. 
Hence $E^n/\phi\Gamma_1\phi^{-1}$ is compact. 
Therefore $\phi\Gamma_1\phi^{-1}$ is a $n$-space group. 

Now $\phi_\ast: \Gamma_1 \to \phi\Gamma_1\phi^{-1}$, defined by $\phi_\ast(\gamma) = \phi\gamma\phi^{-1}$,  
is an isomorphism that maps the normal subgroup $\Nu_1$ to the normal subgroup 
$\phi\Nu_1\phi^{-1}$ of $\phi\Gamma_1\phi^{-1}$, 
and $\phi\Gamma_1\phi^{-1}/\phi\Nu_1\phi^{-1}$ is isomorphic to $\Gamma_1/\Nu_1$. 
Hence $\phi\Gamma_1\phi^{-1}/\phi\Nu_1\phi^{-1}$ is a space group. 
Therefore $\phi\Nu_1\phi^{-1}$ is a complete normal subgroup of $\phi\Gamma_1\phi^{-1}$ by Theorem 5 of \cite{R-T}. 

Now suppose $\nu = a+A \in \Nu_1$.  Then $a' = 0$ and $A' = I'$, and so $\nu' = I'$.  
Hence $\ov{\phi\nu\phi^{-1}} = \alpha\ov \nu\alpha^{-1}$ and $(\phi\nu\phi^{-1})' = I'$. 
As $\alpha\ov\Nu_1\alpha^{-1} = \ov\Nu_2$, we have that  $\phi\Nu_1\phi^{-1} = \Nu_2$.  
Moreover, as $\beta\Gamma_1'\beta^{-1}=\Gamma_2'$, we have that $(\phi\Gamma_1\phi^{-1})' = \Gamma_2'$. 

Let $\Xi: \phi\Gamma_1\phi^{-1}/\Nu_2 \to \mathrm{Isom}(V_2/\Nu_2)$ be the homomorphism induced by the action 
of $\phi\Gamma_1\phi^{-1}/\Nu_2$ on $V_2/\Nu_2$.  Let $\gamma =b+B \in \Gamma_1$, 
and let $\Rho:\phi\Gamma_1\phi^{-1}/\Nu_2 \to \Gamma_2'$ 
be the isomorphism defined by $\Rho(\Nu_2\phi\gamma\phi^{-1}) = (\phi\gamma\phi^{-1})'$. 
Then we have that 
\begin{eqnarray*} \Xi\Rho^{-1}(\beta\gamma'\beta^{-1})
& = & \Xi\Rho^{-1}((\phi\gamma\phi^{-1})') \\
& = & \Xi(\Nu_2\phi\gamma\phi^{-1}) \\
& = & (\ov{\phi\gamma\phi^{-1}})_\star \\
& = & ((Db'+\ov{I})\alpha\ov{\gamma}{\alpha}^{-1})_\star \\
& = & (Db'+\ov{I})_\star\alpha_\star\ov{\gamma}_\star{\alpha}^{-1}_\star \\
& = & (Dp_1(\Nu_1(b+B)))_\star\alpha_\sharp(\ov\gamma_\star) \\
& = & (Dp_1(\Nu_1\gamma))_\star\alpha_\sharp(\Xi_1(\Nu_1\gamma)) \\
& = & \Xi_2\Rho_2^{-1}\beta_\ast\Rho_1(\Nu_1\gamma) \ \
= \ \ \Xi_2\Rho_2^{-1}(\beta\gamma'\beta^{-1}). 
\end{eqnarray*} 
Hence we have that $\Xi\Rho^{-1} = \Xi_2\Rho_2^{-1}$.  
Therefore $\phi\Gamma_1\phi^{-1} = \Gamma_2$ by Theorem 2. 
Thus $\phi(\Gamma_1,\Nu_1)\phi^{-1} = (\Gamma_2,\Nu_2)$. 

Let $\gamma = b+B$ and $\gamma_1 = b_1+B_1$ be elements of  $\Gamma_1$.  
Then we have that 
\begin{eqnarray*} (Dp_1)_\star(\Nu_1\gamma\Nu_1\gamma_1) 
& = &(Dp_1)_\star(\Nu_1(b+Bb_1+BB_1))  \\
& = & (D(b+Bb_1)'+\ov I)_\star  \\
& = &(D(b'+B'b_1')+\ov I)_\star \\
& = &(Db'+DB'b_1'+\ov I)_\star\\
& = & (Db'+\ov C \ov B \ov C^{-1}Db_1'+\ov I)_\star \\
& = &(Db'+\ov I)_\star (\ov C \ov B \ov C^{-1}Db_1'+\ov I)_\star \\
& = & (Db'+\ov I)_\star (\Nu_1(b+B))(Db_1'+\ov I)_\star\\
& = & (Dp_1)_\star(\Nu_1\gamma) ((\Nu_1\gamma)(Dp_1)_\star(\Nu_1\gamma_1)). 
\end{eqnarray*} 
Therefore $(Dp_1)_\star : \Gamma_1/\Nu_1 \to \mathcal{K}_2$ is a crossed homomorphism. 
Thus statements (1) and (2) are equivalent. 

The equation $\Xi_2\Rho_2^{-1}\beta_\ast\Rho_1 = (Dp_1)_\star\alpha_\sharp\Xi_1$ 
is equivalent to the equation 
$$\alpha_\sharp^{-1} \Xi_2\Rho_2^{-1}\beta_\ast\Rho_1 = \alpha_\sharp^{-1} (Dp_1)_\star\alpha_\sharp\Xi_1.$$
Observe that 
\begin{eqnarray*} \alpha_\sharp^{-1} (Dp_1)_\star\alpha_\sharp\Xi_1(\Nu_1\gamma) 
& = & \alpha_\sharp^{-1}( (Dp_1)_\star (\Nu_1\gamma) \alpha_\sharp \Xi_1(\Nu_1\gamma)) \\
& = & \alpha_\sharp^{-1}( (Db'+ \ov I)_\star \alpha_\sharp \ov \gamma_\star) \\
& = & \alpha_\sharp^{-1}( (Db'+ \ov I)_\star \alpha_\star \ov \gamma_\star \alpha_\star^{-1}) \\
& = & \alpha_\star^{-1}(Db'+ \ov I)_\star \alpha_\star \ov \gamma_\star \alpha_\star^{-1} \alpha_\star \\
& = & (\alpha^{-1}(Db' + \ov I)\alpha)_\star\ov\gamma_\star \\
& = & (\ov C^{-1}Db'+\ov I)_\star \Xi_1(\Nu_1\gamma) \\
& = & (\ov C^{-1}Dp_1)_\star(\Nu_1\gamma)\Xi_1(\Nu_1\gamma).
\end{eqnarray*}
Hence we have that $\alpha_\sharp^{-1} (Dp_1)_\star\alpha_\sharp\Xi_1 = (\ov C^{-1}Dp_1)_\star\Xi_1$. 
By the same argument as with $(Dp_1)_\star : \Gamma_1/\Nu_1 \to \mathcal{K}_2$, 
we have that $(\ov C^{-1}Dp_1)_\star : \Gamma_1/\Nu_1 \to \mathcal{K}_1$ is a crossed homomorphism. 
Thus (2) and (3) are equivalent. 
\end{proof}

\section{Outer Automorphism Groups of Space Groups}  

Through this section, let $m$ be a positive integer less than $n$. 
Let $\Mu$ be an $m$-space group and 
let $\Delta$ be an $(n-m)$-space group.  

\vspace{.15in}
\noindent{\bf Definition:} Define $\mathrm{Iso}(\Delta,\Mu)$ to be the set of isomorphism classes 
of pairs $(\Gamma, \Nu)$ 
where $\Nu$ is a complete normal subgroup of an $n$-space group $\Gamma$ 
such that $\Nu$ is isomorphic to $\Mu$ and $\Gamma/\Nu$ is isomorphic to $\Delta$. 
We denote the isomorphism class of a pair $(\Gamma,\Nu)$ by $[\Gamma,\Nu]$. 

\vspace{.15in}
Let $\Nu$ be a complete normal subgroup of an $n$-space group $\Gamma$,  
and let $\mathrm{Out}_E(\Nu)$ be the Euclidean outer automorphism group of $\Nu$ 
defined in \S 4 of \cite{R-T-Isom}. 
The group $\mathrm{Out}_E(\Nu)$ is finite by Theorem 2 of \cite{R-T-Isom}. 
The action of $\Gamma$ on $\Nu$ by conjugation induces a homomorphism 
$$\mathcal{O}: \Gamma/\Nu \to \mathrm{Out}_E(\Nu)$$
defined by $\mathcal{O}(\Nu\gamma) = \gamma_\ast\mathrm{Inn}(\Nu)$ where $\gamma_\ast(\nu) = \gamma\nu\gamma^{-1}$ 
for each $\gamma \in \Gamma$ and $\nu\in\Nu$. 
Let  $\alpha: \Nu_1 \to \Nu_2$ be an isomorphism.  Then $\alpha$ induces an isomorphism
$$\alpha_\#: \mathrm{Out}(\Nu_1) \to \mathrm{Out}(\Nu_2)$$
defined by $\alpha_\#(\zeta\mathrm{Inn}(\Nu_1) )= \alpha\zeta\alpha^{-1}\mathrm{Inn}(\Nu_2)$ for each $\zeta\in\mathrm{Aut}(\Nu_1)$. 

\begin{lemma}  
Let $\Nu_i$ be a complete normal subgroup of an $n$-space group $\Gamma_i$ for $i = 1,2$.  
Let $\mathcal{O}_i:\Gamma_i/\Nu_i \to \mathrm{Out}_E(\Nu_i)$ be the homomorphism 
induced by the action of $\Gamma_i$ on $\Nu_i$ by conjugation for $i = 1, 2$, and  
let $\alpha:\Nu_1\to \Nu_2$ and $\phi: \Gamma_1\to \Gamma_2$ and $\beta:\Gamma_1/\Nu_1 \to \Gamma_2/\Nu_2$ be isomorphisms  
such that the following diagram commutes
\[\begin{array}{ccccccccc}
1 & \to & \Nu_1 & \rightarrow & \Gamma_1 & \rightarrow & \Gamma_1/\Nu_1 & \to & 1 \\
    &       & \hspace{.12in} \downarrow \, \alpha &    & \hspace{.12in}\downarrow\, \phi  & & \hspace{.12in} \downarrow\, \beta & \\
1 & \to  &  \Nu_2 & \rightarrow  & \Gamma_2 & \rightarrow & \Gamma_2/\Nu_2 & \to & 1,  
\end{array}\] 
where the horizontal maps are inclusions and projections, then $\mathcal{O}_2 = \alpha_\#\mathcal{O}_1\beta^{-1}$. 
\end{lemma}
\begin{proof}  
Let $\gamma\in \Gamma_1$.  Then we have that 
$$\mathcal{O}_2(\Nu_2\phi(\gamma)) = \phi(\gamma)_\ast\mathrm{Inn}(\Nu_2),$$
whereas
\begin{eqnarray*}
\alpha_\#\mathcal{O}_1\beta^{-1}(\Nu_2\phi(\gamma)) & =  & \alpha_\#\mathcal{O}_1(\Nu_1\gamma) \\
                                              & = & \alpha_\#(\gamma_\ast\mathrm{Inn}(\Nu_1)) \ \ = \ \ \alpha\gamma_\ast\alpha^{-1}\mathrm{Inn}(\Nu_2).
\end{eqnarray*}
If $\nu \in \Nu$,  then 
\begin{eqnarray*}
\alpha\gamma_\ast\alpha^{-1}(\nu) & = & \alpha\gamma_\ast\alpha^{-1}(\nu) \\
						       & = & \alpha(\gamma\alpha^{-1}(\nu)\gamma^{-1}) \\
						       & = & \phi(\gamma\phi^{-1}(\nu)\gamma^{-1}) \\
						       & = & \phi(\gamma)\nu\phi(\gamma)^{-1} \ \ = \ \ \phi(\gamma)_\ast(\nu). 
\end{eqnarray*}
Hence $\alpha\gamma_\ast\alpha^{-1} = \phi(\gamma)_\ast$.  
Therefore $\mathcal{O}_2 = \alpha_\#\mathcal{O}_1\beta^{-1}$. 
\end{proof}

\noindent{\bf Definition:} 
Define $\mathrm{Hom}_f(\Delta,\mathrm{Out}(\Mu))$ to be the set of all homomorphisms from $\Delta$ to $\mathrm{Out}(\Mu)$ 
that have finite image. 

\vspace{.15in}
The group $\mathrm{Out}(\Mu)$ acts on the left of $\mathrm{Hom}_f(\Delta,\mathrm{Out}(\Mu))$ by conjugation, 
that is, if $g\in \mathrm{Out}(\Mu)$ and $\eta\in \mathrm{Hom}_f(\Delta,\mathrm{Out}(\Mu))$, 
then $g \eta = g_\ast\eta$ where $g_\ast: \mathrm{Out}(\Mu) \to \mathrm{Out}(\Mu)$ is defined by $g_\ast(h) = ghg^{-1}$. 
Let $\mathrm{Out}(\Mu)\backslash\mathrm{Hom}_f(\Delta,\mathrm{Out}(\Mu))$ be the set of $\mathrm{Out}(\Mu)$-orbits. 
The group $\mathrm{Aut}(\Delta)$ acts on the right of $\mathrm{Hom}_f(\Delta,\mathrm{Out}(\Mu))$ 
by composition of homomorphisms. 
If $\beta\in \mathrm{Aut}(\Delta)$ and  $\eta\in \mathrm{Hom}_f(\Delta,\mathrm{Out}(\Mu))$ and $g\in \mathrm{Out}(\Mu)$, 
then 
$$(g\eta)\beta = (g_\ast\eta)\beta = g_\ast(\eta\beta) = g(\eta\beta).$$
Hence $\mathrm{Aut}(\Delta)$ acts on the right of 
$\mathrm{Out}(\Mu)\backslash\mathrm{Hom}_f(\Delta,\mathrm{Out}(\Mu))$ 
by 
$$(\mathrm{Out}(\Mu)\eta)\beta = \mathrm{Out}(\Mu)(\eta\beta).$$
Let $\delta, \epsilon \in \Delta$ and $\eta\in \mathrm{Hom}_f(\Delta,\mathrm{Out}(\Mu))$.  
Then we have that
$$\eta\delta_\ast(\epsilon) = \eta(\delta\epsilon\delta^{-1}) =\eta(\delta)\eta(\epsilon)\eta(\delta)^{-1}= \eta(\delta)_\ast\eta(\epsilon) 
= (\eta(\delta)\eta)(\epsilon). $$
Hence $\eta\delta_\ast = \eta(\delta)\eta$.  Therefore $\mathrm{Inn}(\Delta)$ acts trivially on
 $\mathrm{Out}(\Mu)\backslash\mathrm{Hom}_f(\Delta,\mathrm{Out}(\Mu))$. 
Hence $\mathrm{Out}(\Delta)$ acts on the right of $\mathrm{Out}(\Mu)\backslash\mathrm{Hom}_f(\Delta,\mathrm{Out}(\Mu))$ 
by
$$(\mathrm{Out}(\Mu)\eta)(\beta\mathrm{Inn}(\Delta)) = \mathrm{Out}(\Mu)(\eta\beta).$$
\noindent{\bf Definition:} Define the set $\mathrm{Out}(\Delta,\Mu)$ by the formula
$$\mathrm{Out}(\Delta,\Mu) = (\mathrm{Out}(\Mu)\backslash\mathrm{Hom}_f(\Delta,\mathrm{Out}(\Mu)))/\mathrm{Out}(\Delta).$$ 
If $\eta\in \mathrm{Hom}_f(\Delta,\mathrm{Out}(\Mu))$, let $[\eta] = (\mathrm{Out}(\Mu)\eta)\mathrm{Out}(\Delta)$ 
be the element of $\mathrm{Out}(\Delta,\Mu)$ determined by $\eta$. 

\vspace{.15in}
Let $(\Gamma,\Nu)$ be a pair such that $[\Gamma,\Nu]\in\mathrm{Iso}(\Delta,\Mu)$. 
Let  $\mathcal{O}: \Gamma/\Nu \to \mathrm{Out}_E(\Nu)$ 
be the homomorphism induced by the action of $\Gamma$ on $\Nu$ by conjugation. 
Let $\alpha: \Nu \to \Mu$ and $\beta:\Delta \to \Gamma/\Nu$ be isomorphisms. 
Then $\alpha_\#\mathcal{O}\beta \in \mathrm{Hom}_f(\Delta,\mathrm{Out}(\Mu))$. 

Let $\alpha':\Nu\to\Mu$ and $\beta':\Delta\to\Gamma/\Nu$ are isomorphisms. 
Observe that 
\begin{eqnarray*}
\alpha'_\#\mathcal{O}\beta' & = & \alpha'_\#\alpha_\#^{-1}\alpha_\#\mathcal{O}\beta\beta^{-1}\beta'  \\
& = & (\alpha'\alpha^{-1})_\#\alpha_\#\mathcal{O}\beta(\beta^{-1}\beta')  \\
& = & (\alpha'\alpha^{-1}\mathrm{Inn}(\Mu))_\ast(\alpha_\#\mathcal{O}\beta(\beta^{-1}\beta'))  \\
& = & (\alpha'\alpha^{-1}\mathrm{Inn}(\Mu))(\alpha_\#\mathcal{O}\beta)(\beta^{-1}\beta').  \\
\end{eqnarray*}
Hence $[\alpha_\#\mathcal{O}\beta]$ in $\mathrm{Out}(\Delta,\Mu)$ does not depend 
on the choice of $\alpha$ and $\beta$,  
and so $(\Gamma,\Nu)$ determines the element $[\alpha_\#\mathcal{O}\beta]$ 
of $\mathrm{Out}(\Delta,\Mu)$ independent of the choice of $\alpha$ and $\beta$. 

Suppose $[\Gamma_i,\Nu_i]\in \mathrm{Iso}(\Delta,\Mu)$ for $i=1,2$, 
and $\phi:(\Gamma_1,\Nu_1) \to (\Gamma_2,\Nu_2)$ is an isomorphism of pairs. 
Let $\alpha:\Nu_1\to \Nu_2$ be the isomorphism obtained by restricting $\phi$, 
and let $\beta: \Gamma_1/\Nu_1\to \Gamma_2/\Nu_2$ be the isomorphism induced by $\phi$. 
Let $\mathcal{O}_i:\Gamma_i/\Nu_i \to \mathrm{Out}_E(\Nu_i)$ be the homomorphism 
induced by the action of $\Gamma_i$ on $\Nu_i$ by conjugation for $i=1,2$. 
Then $\mathcal{O}_2 = \alpha_\#\mathcal{O}_1\beta^{-1}$ by Lemma 17.  
Let $\alpha_1:\Nu_1 \to \Mu$ and $\beta_1:\Delta \to \Gamma_1/\Nu_1$ be isomorphisms. 
Let $\alpha_2 = \alpha_1\alpha^{-1}$ and $\beta_2= \beta\beta_1$. 
Then we have 
$$(\alpha_2)_\#\mathcal{O}_2\beta_2 
 =  (\alpha_1\alpha^{-1})_\#\alpha_\#\mathcal{O}_1\beta^{-1}\beta_2 \\
 = (\alpha_1)_\#\mathcal{O}_1\beta_1. $$
Hence $(\Gamma_1,\Nu_1)$ and $(\Gamma_2,\Nu_2)$ determine the same element of $\mathrm{Out}(\Delta,\Mu)$. 
Therefore there is a function 
$$\omega:\mathrm{Iso}(\Delta,\Mu) \to \mathrm{Out}(\Delta,\Mu)$$
defined by $\omega([\Gamma,\Nu]) = [\alpha_\#\mathcal{O}\beta]$ for any choice of isomorphisms $\alpha:\Nu\to\Mu$ 
and $\beta: \Delta\to \Gamma/\Nu$. 

\begin{lemma} 
The set $\mathrm{Out}(\Delta,\Mu)$ is finite. 
\end{lemma}
\begin{proof}
The group $\mathrm{Out}(\Mu)$ is arithmetic by Theorem 1.1 of \cite{B-G}. 
Hence $\mathrm{Out}(\Mu)$ has only finitely many conjugacy classes of finite subgroups, cf.\,\S 5 of \cite{Borel}. 
As $\Delta$ is finitely generated there are only finitely many homomorphisms from $\Delta$ to a finite group $G$. 
Therefore $\mathrm{Out}(\Mu)\backslash\mathrm{Hom}_f(\Delta,\mathrm{Out}(\Mu))$ is finite. 
Hence $\mathrm{Out}(\Delta,\Mu) $ is finite. 
\end{proof}


\section{Fiber Cohomology Classes}  

Consider $\omega:\mathrm{Iso}(\Delta,\Mu) \to \mathrm{Out}(\Delta,\Mu)$.  
Suppose $[\Gamma_1,\Nu_1]$ and $[\Gamma_2,\Nu_2]$ are in the same fiber of $\omega$. 
We want to define a class $[\Gamma_1,\Nu_1;\Gamma_2,\Nu_2;\alpha,\beta]$ in the cohomology group $H^1(\Gamma_1/\Nu_1, \mathcal{K}_1)$ 
where $\Gamma_1/\Nu_1$ acts on $\mathcal{K}_1$ by $\Nu_1(b+B)(v+\ov I)_\star = (\ov B v + \ov I)_\star$. 

Let $\alpha_i: \Nu_i \to \Mu$ and $\beta_i: \Delta \to \Gamma_i/\Nu_i$ be isomorphisms for $i = 1,2$. 
Let $\mathcal{O}_i:\Gamma_i/\Nu_i \to \mathrm{Out}_E(\Nu_i)$ be the homomorphism 
induced by the action of $\Gamma_i$ on $\Nu_i$ by conjugation for $i=1,2$. 
As $\omega([\Gamma_1,\Nu_1]) = \omega([\Gamma_2,\Nu_2])$, we have that 
$[(\alpha_1)_\#\mathcal{O}_1\beta_1]=[(\alpha_2)_\#\mathcal{O}_2\beta_2]$. 
Then there exists $\alpha_0$ in $\mathrm{Aut}(\Mu)$ and $\beta_0$ in $\mathrm{Aut}(\Delta)$ such that 
$(\alpha_1)_\#\mathcal{O}_1\beta_1 = (\alpha_0)_\#(\alpha_2)_\#\mathcal{O}_2\beta_2\beta_0$. 
We have that 
$$\mathcal{O}_1 = (\alpha_1^{-1}\alpha_0\alpha_2)_\#\mathcal{O}_2\beta_2\beta_0\beta_1^{-1}.$$ 

Let $\alpha: \Nu_1 \to \Nu_2$ be the isomorphism $\alpha_2^{-1}\alpha_0^{-1}\alpha_1$, 
and let $\beta: \Gamma_1/\Nu_1 \to \Gamma_2/\Nu_2$ be the isomorphism  $\beta_2\beta_0\beta_1^{-1}$. 
Then $\mathcal{O}_1 = \alpha^{-1}_\#\mathcal{O}_2\beta$.
Now $\alpha$ induces an isomorphism $\ov\alpha: \ov\Nu_1 \to \ov\Nu_2$ defined by $\ov\alpha(\ov\nu) = \ov{\alpha(\nu)}$ for each $\nu$ in $\Nu_1$. 
Let $V_i = \mathrm{Span}(\Nu_i)$ for $i = 1, 2$, 
and let $\tilde\alpha: V_1 \to V_2$ be an affinity such that $\tilde\alpha\ov\Nu_1\tilde\alpha^{-1} = \ov\Nu_2$ and $\tilde\alpha_\ast = \ov\alpha$, 
that is, $\tilde\alpha \ov\nu \tilde\alpha^{-1} = \ov\alpha(\ov\nu)$ for each $\nu$ in $\Nu_1$. 

Let $\Xi_i: \Gamma_i/\Nu_i \to \mathrm{Isom}(V_i/\Nu_i)$ be the homomorphism induced by the action of $\Gamma_i/\Nu_i$ on $V_i/\Nu_i$ 
for $i = 1, 2$.  
Let $\Omega_i : \mathrm{Isom}(V_i/\Nu_i) \to \mathrm{Out}_E(\Nu_i)$ be defined so that $\Omega_i(\zeta) = \mathrm{Inn}(\Nu_i)\hat\zeta_\ast$ 
where $\hat\zeta$ is an isometry of $V_i$ that lifts $\zeta$ and $\hat\zeta_\ast$ is the automorphism of $\Nu_i$ defined 
by $\ov{\hat\zeta_\ast(\nu)} = \hat\zeta\ov \nu \hat\zeta^{-1}$ for $i = 1, 2$.  
Then we have that $\Omega_i\Xi_i = \mathcal{O}_i$ for $i = 1, 2$. 
By Lemma 10 of \cite{R-T-Isom}, we have that 
$$\Omega_1\Xi_1 = \alpha_\#^{-1}\Omega_2\Xi_2\beta = 
\Omega_1(\tilde\alpha_\sharp^{-1}\Xi_2\beta).$$ 

Let $\phi: \Gamma_1/\Nu_1 \to \mathrm{Aff}(V_1/\Nu_1)$ and $\psi: \Gamma_1/\Nu_1 \to \mathrm{Aff}(V_1/\Nu_1)$ 
be the homomorphisms defined by $\phi = \tilde\alpha_\sharp^{-1}\Xi_2\beta$ and $\psi = \Xi_1$. 
Then we have that $\phi(g)\psi(g)^{-1}$ is in $\mathcal{K}_1$ for each $g$ in $\Gamma_1/\Nu_1$ by Theorem 3 of \cite{R-T-Isom}. 
As $\psi$ takes values in $\mathrm{Isom}(V_1/\Nu_1)$ and $\mathcal{K}_1$ is a subgroup of $\mathrm{Isom}(V_1/\Nu_1)$, 
we have that $\phi$ takes values in $\mathrm{Isom}(V_1/\Nu_1)$. 

Let $g, h$ be in $\Gamma_1/\Nu_1$, then we have that 
\begin{eqnarray*} 
\phi(gh)\psi(gh)^{-1} & = & \phi(g)\phi(h)(\psi(g)\psi(h))^{-1} \\
& = & \phi(g)\phi(h)\psi(h)^{-1}\psi(g)^{-1} \\
& = & \phi(g)\psi(g)^{-1}\psi(g)\phi(h)\psi(h)^{-1}\psi(g)^{-1} \\
& = & (\phi(g)\psi(g)^{-1})\psi(g)_\ast(\phi(h)\psi(h)^{-1})
\end{eqnarray*}
with 
$$\psi(\Nu_1(b+B))_\ast(v+\ov I) = (\ov b+ \ov B)_\star (v+ \ov I)_\star (\ov b + \ov B)_\star^{-1} = (\ov Bv + \ov I)_\star.$$
Hence the function $\phi\psi^{-1}: \Gamma_1/\Nu_1 \to \mathcal{K}_1$ is a crossed homomorphism, 
and so determines a class $[\Gamma_1,\Nu_1; \Gamma_2,\Nu_2;\alpha,\beta]$ in $H^1(\Gamma_1/\Nu_1, \mathcal{K}_1)$, 
cf.\,p.\,105 of \cite{M}. 

Let $[\Gamma,\Nu]$ be a class in $\mathrm{Isom}(\Delta, \Mu)$, 
and let $V = \mathrm{Span}(\Nu)$. 
Let $\mathcal{C}$ be the centralizer of $\ov\Nu$ in $\mathrm{Aff}(V)$. 
By Lemmas 6 and 8 of \cite{R-T-Isom}, we have that 
$$\mathcal{C}=\{v+\ov I: v \in \mathrm{Span}(Z(\Nu))\}.$$ 
The group $\Gamma/\Nu$ acts on $\mathcal{C}$ by $(\Nu(b+B))(v+\ov I) = \ov Bv + \ov I$
and $\Gamma/\Nu$ acts on $Z(\Nu)$ by $(\Nu(b+B))(u+ I) = Bu +  I$. 
We have a short exact sequence of $(\Gamma/\Nu)$-modules
$$0 \to Z(\Nu)\ {\buildrel \iota \over \longrightarrow}\ \mathcal{C}\ {\buildrel \kappa \over \longrightarrow}\ \mathcal{K} \to 0$$
where $\iota(u+I ) = u+\ov I$ and $\kappa(v+ \ov I ) = (v+\ov I)_\star$. 

Let $\Tau$ be the group of translations of $\Gamma$.  
Then $\Tau\Nu/\Nu$ is a normal subgroup of $\Gamma/\Nu$ of finite index 
and $\Tau\Nu/\Nu$ is a subgroup of the group of translations of $\Gamma/\Nu$ of finite index by Theorem 16 of \cite{R-T}. 
The group $\Gamma/\Nu$ acts on the abelian group $\Tau\Nu/\Nu$ by  $(\Nu(b+B))(a+ I) = \Nu(Ba +  I)$, 
and so $\Tau\Nu/\Nu$ is a $(\Gamma/\Nu)$-module. 
Moreover the group $\Tau\Nu/\Nu$ acts trivially on $\mathcal{C}$. 

\begin{lemma} 
Let $f: \Gamma/\Nu \to \mathcal{C}$ be a crossed homomorphism, 
and let $f_{res}: \Tau\Nu/\Nu \to \mathcal{C}$ be the restriction of $f$. 
Then $f_{res}$ is a homomorphism of $(\Gamma/\Nu)$-modules 
and the class of $f$ in $H^1(\Gamma/\Nu, \mathcal{C})$ is completely determined by $f_{res}$. 
\end{lemma}
\begin{proof}
According to \cite{M}, p.\,354, we have an exact sequence of homomorphisms
$$H^1(\Gamma/\Tau\Nu,\mathcal{C})\ {\buildrel inf \over \longrightarrow}\ H^1(\Gamma/\Nu,\mathcal{C})
\ {\buildrel res \over \longrightarrow}\ H^1(\Tau\Nu/\Nu,\mathcal{C})^{\Gamma/\Nu} \to H^2(\Gamma/\Tau\Nu,\mathcal{C}).$$
The group $\Gamma/\Tau\Nu$ is finite and $\mathcal{C}$ is a torsion-free, divisible, abelian group, 
and so $H^i(\Gamma/\Tau\Nu,\mathcal{C})= 0$ for $i =1, 2$ by Corollary IV.5.4 of \cite{M}.  
Hence 
$$res: H^1(\Gamma/\Nu,\mathcal{C}) \to H^1(\Tau\Nu/\Nu,\mathcal{C})^{\Gamma/\Nu}$$ 
is an isomorphism. 
Here $res([f]) = [f_{res}]$. 
By the Universal Coefficients Theorem (p.\,77 of \cite{M}), 
we have that 
$$H^1(\Tau\Nu/\Nu,\mathcal{C})^{\Gamma/\Nu} = \mathrm{Hom}(\Tau\Nu/\Nu,\mathcal{C})^{\Gamma/\Nu}.$$
Here $\Gamma/\Nu$ acts on $\mathrm{Hom}(\Tau\Nu/\Nu,\mathcal{C})$ by $((\Nu\gamma)h)(x) = (\Nu\gamma) h(\Nu\gamma^{-1}x)$ for each $\gamma\in\Gamma$, 
homomorphism $h: \Tau\Nu/\Nu \to \mathcal{C}$, and element $x \in \Tau\Nu/\Nu$. 
Therefore, we have that 
$$\mathrm{Hom}(\Tau\Nu/\Nu,\mathcal{C})^{\Gamma/\Nu} = \mathrm{Hom}_{\Gamma/\Nu}(\Tau\Nu/\Nu,\mathcal{C}).$$
Hence $[f_{res}] =\{f_{res}\}$ and $f_{res}$ is a homomorphism of $(\Gamma/\Nu)$-modules. 
Therefore the class of $f$ in $H^1(\Gamma/\Nu, \mathcal{C})$ is completely determined by $f_{res}$. 
\end{proof}

\begin{lemma} 
Suppose that $\omega([\Gamma_1,\Nu_1]) = \omega([\Gamma_2,\Nu_2])$ with $\mathcal{O}_1 = \alpha_\#^{-1}\mathcal{O}_2\beta$. 
If the class $[\Gamma_1,\Nu_1; \Gamma_2,\Nu_2;\alpha,\beta]$ is in the image of $\kappa_\ast: H^1(\Gamma_1/\Nu_1,\mathcal{C}_1) \to H^1(\Gamma_1/\Nu_1,\mathcal{K}_1)$, 
then $[\Gamma_1,\Nu_1] = [\Gamma_2,\Nu_2]$. 
\end{lemma}
\begin{proof}
Suppose that $[\Gamma_1,\Nu_1; \Gamma_2,\Nu_2;\alpha,\beta]$ is in the image of 
$\kappa_\ast: H^1(\Gamma_1/\Nu_1,\mathcal{C}_1) \to H^1(\Gamma_1/\Nu_1,\mathcal{K}_1)$. 
Then there is a crossed homomorphism $f: \Gamma_1/\Nu_1 \to \mathcal{C}_1$ such that 
$$\kappa_\ast([f]) = [\Gamma_1,\Nu_1; \Gamma_2,\Nu_2;\alpha,\beta].$$
By Lemma 5, the cohomology class $[f]$ is completely determined by the restriction $(\Gamma_1/\Nu_1)$-module homomorphism 
$f_{res}: \Tau_1\Nu_1/\Nu_1 \to C_1$.  

To simplify notation, replace $\Gamma_1/\Nu_1$ with $\Gamma_1'$.  
Then $\Tau_1\Nu_1/\Nu_1$ corresponds to $T_1' = \{b'+I': b+I \in T_1\}$.   
Moreover $\Gamma_1/\Nu_1$ acts on $T_1'$ by $(\Nu_1(b+B))(b'+I') = B'b'+I'$ 
and $f_{res}$ corresponds to a homomorphism of $(\Gamma_1/\Nu_1)$-modules $f_{res}': T_1' \to \mathcal{C}_1$. 

The group $\Tau_1\Nu_1/\Nu_1$ has finite index in the group of translations of $\Gamma_1/\Nu_1$, 
and so $T_1'$ has finite index in the group of translations of $\Gamma_1'$. 
Hence $\{b': b+I \in T_1\}$ contains a basis of the vector space $V_1^\perp$. 
Therefore $f_{res}': T_1' \to C_1$ induces a linear transformation $L: V_1^\perp \to \mathrm{Span}(Z(\Nu_1))$ 
such that $f_{res}'(b'+I') = L(b')+\ov I$ for each $b+I$ in $T_1$ and if $b+ B$ is in $\Gamma_1$, then $LB' = \ov BL$. 

Consider the function $h: \Gamma_1/\Nu_1 \to \mathcal{C}_1$ defined by $h(\Nu_1(b+ B)) =L(b')+\ov I$. 
Then $h$ is a crossed homomorphism. 
If $b+I$ is in $T_1$, then 
$$h_{res}(\Nu_1(b+I)) = L(b')+\ov I = f_{res}'(b'+I') = f_{res}(\Nu_1(b+I)).$$
Hence $h_{res} = f_{res}$. 
Therefore $[h] = [f]$ in $H^1(\Gamma_1/\Nu_1, \mathcal{C}_1)$ by Lemma 5. 

Let $\tilde\alpha: V_1 \to V_2$ be the affinity defined above and write $\tilde\alpha = \ov c + \ov C$ 
with $\ov c \in V_2$ and $\ov C: V_1 \to V_2$ a linear isomorphism. 
Define a linear transformation $D: V_1^\perp \to \mathrm{Span}(\Nu_2)$ by $D = \ov C L$. 
If $b+B \in \Gamma_1$, then $DB' = \ov C\ov B\ov C^{-1} D$. 
Let $p_1: \Gamma_1/\Nu_1 \to V_1^\perp$ be the crossed homomorphism defined by $p_1(\Nu_1(b+B)) = b'$. 
Then $h(\Nu_1\gamma) = \ov C^{-1}D p_1(\Nu_1\gamma)+ \ov I$. 
Observe that $\kappa_\ast([h]) = [h_\star]$ where $h_\star$ is defined 
by $h_\star(\Nu_1\gamma) = (h(\Nu_1\gamma))_\star$. 
Thus $h_\star = (\ov C^{-1}D p_1)_\star$ as defined in Theorem 3(3). 

Let $v$ be in $\mathrm{Span}(Z(\Nu_1))$, and let $f_v: \Gamma_1/\Nu_1 \to \mathcal{K}_1$ be the principal crossed homomorphism 
determined by $(v+ I)_\star$. 
Then we have that
\begin{eqnarray*}
f_v(\Nu_1(b+B)) & =  & (\Nu_1(b+B))(v+\ov I)_\star (v+\ov I)_\star^{-1} \\
& = & (\ov Bv + \ov I)_\star(-v+\ov I)_\star\ \ =\ \ (\ov Bv - v +\ov I)_\star.
\end{eqnarray*}

Now we have that $[h_\star] = [\Gamma_1,\Nu_1; \Gamma_2,\Nu_2;\alpha,\beta]$ in $H^1(\Gamma_1/\Nu_1,\mathcal{K}_1)$. 
Hence there exists $v$ in $\mathrm{Span}(Z(\Nu_1))$ such that 
$$(\tilde\alpha_\sharp^{-1}\Xi_2\beta)\Xi_1^{-1}f_v = (\ov C^{-1}D p_1)_\star.$$

Let $\tilde \alpha_v: V_1 \to V_2$ be the affinity defined by $\tilde \alpha_v = \tilde\alpha(v+\ov I)$. 
Then $\tilde\alpha_v\ov N_1\tilde\alpha_v^{-1}= \ov N_2$ and $(\tilde\alpha_v)_\ast = \tilde\alpha_\ast$ by Lemma 6 of \cite{R-T-Isom}, 
and so $(\tilde\alpha_v)_\ast = \ov \alpha$.
Observe that 
\begin{eqnarray*}
\lefteqn{\big((\tilde\alpha_\sharp^{-1}\Xi_2\beta)\Xi_1^{-1}f_v\big)(\Nu_1(b+B))} \\ 
& =  & (\tilde\alpha_\sharp^{-1}\Xi_2\beta)(\Nu_1(b+B))\Xi_1^{-1}(\Nu_1(b+B)) f_v(\Nu_1(b+B))  \\
& = & \tilde\alpha_\star^{-1}(\Xi_2\beta)(\Nu_1(b+B))\tilde\alpha_\star (\ov b + \ov B)_\star^{-1}(\ov Bv - v +\ov I)_\star \\
& = & \tilde\alpha_\star^{-1}(\Xi_2\beta)(\Nu_1(b+B))\tilde\alpha_\star (\ov b + \ov B)_\star^{-1}(v -\ov Bv +\ov I)_\star^{-1} \\
& = & \tilde\alpha_\star^{-1}(\Xi_2\beta)(\Nu_1(b+B))\tilde\alpha_\star (v -\ov Bv +\ov b + \ov B)_\star^{-1} \\
& = & \tilde\alpha_\star^{-1}(\Xi_2\beta)(\Nu_1(b+B))\tilde\alpha_\star ((v +\ov I)(\ov b + \ov B)(-v+I))_\star^{-1} \\
& = & \tilde\alpha_\star^{-1}(\Xi_2\beta)(\Nu_1(b+B))\tilde\alpha_\star (v +\ov I)_\star (\ov b + \ov B)_\star^{-1}(-v+I)_\star \\
& = & (-v+I)_\star\tilde\alpha_\star^{-1}(\Xi_2\beta)(\Nu_1(b+B))\tilde\alpha_\star (v +\ov I)_\star (\ov b + \ov B)_\star^{-1} \\
& = & (\tilde\alpha_v)_\star^{-1}(\Xi_2\beta)(\Nu_1(b+B))(\tilde\alpha_v)_\star (\ov b + \ov B)_\star^{-1} \\
& = & \big(((\tilde\alpha_v)_\sharp^{-1}\Xi_2\beta)\Xi_1^{-1}\big)(\Nu_1(b+B))
\end{eqnarray*}
Hence we have 
$$(\tilde\alpha_\sharp^{-1}\Xi_2\beta)\Xi_1^{-1}f_v = ((\tilde\alpha_v)_\sharp^{-1}\Xi_2\beta)\Xi_1^{-1}.$$
Thus we have that 
$$(\tilde\alpha_v)_\sharp^{-1}\Xi_2\beta = (\ov C^{-1}D p_1)_\star\Xi_1.$$

Let $\Rho_i:\Gamma_i/\Nu_i \to \Gamma_i'$ be the isomorphism defined by $\Rho_i(\Nu_i\gamma) = \gamma'$ for each $i = 1,2$.  
Let $\beta': \Gamma_1' \to \Gamma_2'$ be the isomorphism so that $\Rho_2^{-1}\beta'\Rho_1 = \beta$. 
Let $\tilde\beta: V_1^\perp \to V_2^\perp$ be an affinity such that $\tilde\beta\Gamma_1'\tilde\beta^{-1} = \Gamma_2'$ and $\tilde\beta_\ast = \beta'$, that is, and $\tilde\beta\gamma'\tilde\beta^{-1}= \beta'(\gamma')$ for each $\gamma$ in $\Gamma_1$. 
Then we have that 
$$(\tilde\alpha_v)_\sharp^{-1}\Xi_2\Rho_2^{-1}\tilde\beta_\ast\Rho_1= (\ov C^{-1}D p_1)_\star\Xi_1.$$ 
Therefore there exists an affinity $\phi = c+ C$ of $E^n$ such that $\phi(\Gamma_1,\Nu_1)\phi^{-1} = (\Gamma_2,\Nu_2)$ 
with $\ov \phi = \tilde\alpha_v$, $\phi' = \tilde\beta$, and $\ov{C'} = D$ by Theorem 3. 
Thus $[\Gamma_1,\Nu_1] = [\Gamma_2,\Nu_2]$. 
\end{proof}

\begin{lemma} 
Suppose that $\omega([\Gamma,\Nu]) = \omega([\Gamma_1,\Nu_1])$ with $\mathcal{O} = (\alpha_1)_\#^{-1}\mathcal{O}_1\beta_1$ 
and that $\omega([\Gamma,\Nu]) = \omega([\Gamma_2,\Nu_2])$ with $\mathcal{O} = (\alpha_2)_\#^{-1}\mathcal{O}_2\beta_2$. 
If $[\Gamma,\Nu; \Gamma_i,\Nu_i;\alpha_i,\beta_i]$ for $i = 1, 2$ 
are in the same coset of the image of $\kappa_\ast: H^1(\Gamma/\Nu,\mathcal{C}) \to H^1(\Gamma/\Nu,\mathcal{K})$ in $H^1(\Gamma/\Nu,\mathcal{K})$, 
then $[\Gamma_1,\Nu_1] = [\Gamma_2,\Nu_2]$. 
\end{lemma}
\begin{proof}
Let $b+B \in \Gamma$, and let $b_1+B_1$ be an element of $\Gamma_1$ such that 
$$\Nu_1(b_1+B_1) = \beta_1(\Nu(b+B)).$$
As  $\mathcal{O} = (\alpha_1)_\#^{-1}\mathcal{O}_1\beta_1$, we have that
\begin{eqnarray*}
(b+B)_\ast \mathrm{Inn}(\Nu)& = & (\alpha_1)_\#^{-1}\mathcal{O}_1\beta_1(\Nu(b+B)) \\
& = & (\alpha_1)_\#^{-1}\mathcal{O}_1(\Nu_1(b_1+B_1)) \\
& = & (\alpha_1)_\#^{-1}((b_1+B_1)_\ast\mathrm{Inn}(\Nu_1)) \\
& = & \alpha_1^{-1}(b_1+B_1)_\ast \alpha_1\mathrm{Inn}(\Nu) \\
& = & (\tilde\alpha_1)_\ast^{-1}(b_1+B_1)_\ast (\tilde\alpha_1)_\ast\mathrm{Inn}(\Nu) \\
& = & (\ov c_1 + \ov C_1)_\ast^{-1}(b_1+B_1)_\ast (\ov c_1 + \ov C_1)_\ast\mathrm{Inn}(\Nu). 
\end{eqnarray*}
The action of $\Gamma/\Nu$ on $\mathcal{C}$ is given by $\Nu(b+B)(u+\ov I) = \ov Bu + \ov I$ 
and is determined by the action of $\Gamma/\Nu$ on $Z(\Nu)$ induced by conjugation. 
Now $\mathrm{Inn}(\Nu)$ acts trivially on $Z(\Nu)$, and so the last computation implies 
that $\ov B = \ov C_1^{-1} \ov B_1\ov C_1$ on $\mathrm{Span}(Z(\Nu))$. 
Hence the pair of isomorphisms 
$$\pi = (\beta_1:\Gamma/\Nu \to \Gamma_1/\Nu_1,(\tilde\alpha_1)^{-1}_\ast:\mathcal{C}_1 \to \mathcal{C})$$ 
is a change of groups isomorphism in the sense of \cite{M} p.\,108. 
Therefore we have an isomorphism $\pi^\ast: H^1(\Gamma_1/\Nu_1,\mathcal{C}_1) \to H^1(\Gamma/\Nu,\mathcal{C})$ 
defined so that if $f_1:\Gamma_1/\Nu_1 \to \mathcal{C}_1$ is a crossed homomorphism, then  $\pi^\ast[f_1] = [\pi^\ast f_1]$ 
where $\pi^\ast f_1: \Gamma/\Nu \to \mathcal{C}$ is the crossed homomorphism defined by 
$\pi^\ast f_1(x)  = (\tilde\alpha_1)^{-1}_\ast(f_1(\beta_1(x)))$. 

Likewise the pair of isomorphisms 
$$\varpi=  (\beta_1:\Gamma/\Nu \to \Gamma_1/\Nu_1,(\tilde\alpha_1)^{-1}_\sharp: \mathcal{K}_1 \to \mathcal{K})$$
is a change of groups isomorphism which induces an isomorphism 
$\varpi^\ast$ such that the following diagram commutes 
$$\begin{array}{ccc}
H^1(\Gamma_1/\Nu_1,\mathcal{C}_1) & {\buildrel \pi^\ast \over \longrightarrow} & H^1(\Gamma/\Nu,\mathcal{C}) \vspace{.05in} \\  
\downarrow (\kappa_1)_\ast &  & \downarrow \kappa_\ast \\
H^1(\Gamma_1/\Nu_1,\mathcal{K}_1) & {\buildrel \varpi^\ast \over \longrightarrow} & H^1(\Gamma/\Nu,\mathcal{K}).
\end{array}$$

Now we have that $\omega([\Gamma_1,\Nu_1]) = \omega([\Gamma_2,\Nu_2])$ 
with $\mathcal{O}_1 = (\alpha_2\alpha_1^{-1})_\#\mathcal{O}_2\beta_2\beta_1^{-1}$. 
Moreover, we have that 
$$[\Gamma_1,\Nu_1;\Gamma_2,\Nu_2;\alpha_2\alpha_1^{-1},\beta_2\beta_1^{-1}] = [((\tilde\alpha_2\tilde\alpha_1^{-1})_\sharp^{-1}\Xi_2\beta_2\beta_1^{-1})\Xi_1^{-1}].$$
For $i = 1, 2$, we have that
$$[\Gamma,\Nu;\Gamma_i,\Nu_i;\alpha_i,\beta_i] = [((\tilde\alpha_i)_\sharp^{-1}\Xi_i\beta_i)\Xi^{-1}].$$
Let $\gamma \in \Gamma$.  Observe that
\begin{eqnarray*}
\lefteqn{\big(((\tilde\alpha_2)_\sharp^{-1}\Xi_2\beta_2)\Xi^{-1}\big)\big(((\tilde\alpha_1)_\sharp^{-1}\Xi_1\beta_1)\Xi^{-1}\big)^{-1}(\Nu\gamma)} \\
& = & ((\tilde\alpha_2)_\sharp^{-1}\Xi_2\beta_2)(\Nu\gamma)\Xi(\Nu\gamma)^{-1}\Xi(\Nu\gamma)\big(((\tilde\alpha_1)_\sharp^{-1}\Xi_1\beta_1)(\Nu\gamma)\big)^{-1} \\
& = & ((\tilde\alpha_2)_\sharp^{-1}\Xi_2\beta_2)(\Nu\gamma)(\tilde\alpha_1)_\sharp^{-1}\big(\Xi_1\beta_1(\Nu\gamma)\big)^{-1} \\
& = & ((\tilde\alpha_2)_\star^{-1}\Xi_2\beta_2)(\Nu\gamma)(\tilde\alpha_2)_\star(\tilde\alpha_1)_\star^{-1}\big(\Xi_1\beta_1(\Nu\gamma)\big)^{-1}(\tilde\alpha_1)_\star \\
& = & (\tilde\alpha_1)_\sharp^{-1}((\tilde\alpha_2\tilde\alpha_1^{-1})_\sharp^{-1}\Xi_2\beta_2)(\Nu\gamma)\big(\Xi_1\beta_1(\Nu\gamma)\big)^{-1} \\ 
& = & (\tilde\alpha_1)_\sharp^{-1}((\tilde\alpha_2\tilde\alpha_1^{-1})_\sharp^{-1}\Xi_2\beta_2\beta_1^{-1}(\beta_1(\Nu\gamma))\big(\Xi_1\beta_1(\Nu\gamma)\big)^{-1} \\ 
& = & \varpi^\ast\big(((\tilde\alpha_2\tilde\alpha_1^{-1})_\sharp^{-1}\Xi_2\beta_2\beta_1^{-1})\Xi_1^{-1}\big)(\Nu\gamma). \\ 
\end{eqnarray*} 
Hence we have that 
$$\varpi^\ast([\Gamma_1,\Nu_1;\Gamma_2,\Nu_2;\alpha_2\alpha_1^{-1},\beta_2\beta_1^{-1}] )=[\Gamma,\Nu;\Gamma_2,\Nu_2;\alpha_2,\beta_2][\Gamma,\Nu;\Gamma_1,\Nu_1;\alpha_1,\beta_1]^{-1}.$$
The right-hand side of the above equation is in the image of $\kappa_\ast: H^1(\Gamma/\Nu,\mathcal{C}) \to H^1(\Gamma/\Nu,\mathcal{K})$. 
Therefore $[\Gamma_1,\Nu_1;\Gamma_2,\Nu_2;\alpha_2\alpha_1^{-1},\beta_2\beta_1^{-1}]$ is in the image of 
$(\kappa_1)_\ast: H^1(\Gamma_1/\Nu_1,\mathcal{C}_1) \to H^1(\Gamma_1/\Nu_1,\mathcal{K}_1)$. 
Hence $[\Gamma_1,\Nu_1] = [\Gamma_2,\Nu_2]$ by Lemma 6. 
\end{proof}

\section{The relative Bieberbach Theorem}

Let $[\Gamma, \Nu]$ be a class in $\mathrm{Isom}(\Delta,\Mu)$,  
and let $\Tau$ be the group of translations of $\Gamma$. 
Then $\Gamma/\Tau\Nu$ is a finite group, since $\Gamma/\Tau$ is finite. 
The group $\Tau\Nu/\Nu$ acts trivially on $Z(\Nu)$, $\Tau\Nu/\Nu$, $\mathcal{C}$, $\mathcal{K}$, 
and so the action of $\Gamma/\Nu$ on $Z(\Nu)$, $\Tau\Nu/\Nu$, $\mathcal{C}$, $\mathcal{K}$ induces 
an action of $\Gamma/\Tau\Nu$ on $Z(\Nu)$, $\Tau\Nu/\Nu$, $\mathcal{C}$, $\mathcal{K}$ making 
$Z(\Nu)$, $\Tau\Nu/\Nu$, $\mathcal{C}$, $\mathcal{K}$ into $(\Gamma/\Tau\Nu)$-modules.  

\begin{lemma}  
The group $H^1(\Gamma/\Tau\Nu, \mathcal{K})$ is finite. 
\end{lemma}
\begin{proof}
The short exact sequence $0 \to Z(\Nu) \to \mathcal{C} \to \mathcal{K} \to 0$ of $(\Gamma/\Tau\Nu)$-modules induces 
an exact sequence of cohomology groups
$$H^1(\Gamma/\Tau\Nu,\mathcal{C}) \to H^1(\Gamma/\Tau\Nu,\mathcal{K}) \to H^2(\Gamma/\Tau\Nu,Z(\Nu)) \to H^2(\Gamma/\Tau\Nu,\mathcal{C}).$$
As explained in the proof of Lemma 5, the outside groups are trivial, and so $H^1(\Gamma/\Tau\Nu,\mathcal{K})$ is isomorphic to $H^2(\Gamma/\Tau\Nu,Z(\Nu))$. 
The group $H^2(\Gamma/\Tau\Nu,Z(\Nu))$ is a torsion group by Proposition IV.5.3 of \cite{M}. 
As $Z(\Nu)$ is a free abelian group of finite rank, the group $H^2(\Gamma/\Tau\Nu,Z(\Nu))$ is finitely generated. 
Hence $H^2(\Gamma/\Tau\Nu,Z(\Nu))$ is finite, and so $H^1(\Gamma/\Tau\Nu, \mathcal{K})$ is finite. 
\end{proof}

\begin{lemma} 
The cokernel of $\kappa_\ast: H^1(\Tau\Nu/\Nu, \mathcal{C})^{\Gamma/\Tau\Nu} \to H^1(\Tau\Nu/\Nu, \mathcal{K})^{\Gamma/\Tau\Nu}$ is finite. 
\end{lemma}
\begin{proof}
By the Universal Coefficients Theorem, we have that 
$$H^1(\Tau\Nu/\Nu, \mathcal{C})^{\Gamma/\Tau\Nu} = \mathrm{Hom}(\Tau\Nu/\Nu,\mathcal{C})^{\Gamma/\Tau\Nu} = \mathrm{Hom}_{\Gamma/\Tau\Nu}(\Tau\Nu/\Nu, \mathcal{C}),$$
$$H^1(\Tau\Nu/\Nu, \mathcal{K})^{\Gamma/\Tau\Nu} = \mathrm{Hom}(\Tau\Nu/\Nu,\mathcal{K})^{\Gamma/\Tau\Nu} = \mathrm{Hom}_{\Gamma/\Tau\Nu}(\Tau\Nu/\Nu, \mathcal{K}).$$
The short exact sequence $0 \to Z(\Nu) \to \mathcal{C} \to \mathcal{K} \to 0$ induces 
an exact sequence 
$$0 \to \mathrm{Hom}(\Tau\Nu/\Nu,Z(\Nu)) \to\mathrm{Hom}(\Tau\Nu/\Nu,\mathcal{C}) \to \mathrm{Hom}(\Tau\Nu/\Nu,\mathcal{K}) \to \mathrm{Ext}(\Tau\Nu/\Nu,Z(\Nu))$$
by Theorem III.3.4 of \cite{M} (with $R = \integers$). 
We have that $\mathrm{Ext}(\Tau\Nu/\Nu,Z(\Nu)) = 0$ by Theorems I.6.3 and III.3.5 of \cite{M}, since $\Tau\Nu/\Nu$ is a free abelian group. 
Hence we have a short exact sequence of  $(\Gamma/\Tau\Nu)$-modules
$$0 \to \mathrm{Hom}(\Tau\Nu/\Nu,Z(\Nu)) \to\mathrm{Hom}(\Tau\Nu/\Nu,\mathcal{C}) \to \mathrm{Hom}(\Tau\Nu/\Nu,\mathcal{K}) \to 0.$$
Hence we have an exact sequence of cohomology groups
$$H^0(\Gamma/\Tau\Nu, \mathrm{Hom}(\Tau\Nu/\Nu,\mathcal{C})) \to H^0(\Gamma/\Tau\Nu, \mathrm{Hom}(\Tau\Nu/\Nu,\mathcal{K}))$$ 
$$\to H^1(\Gamma/\Tau\Nu, \mathrm{Hom}(\Tau\Nu/\Nu,Z(\Nu))),$$
which is equivalent to an exact sequence
$$\mathrm{Hom}_{\Gamma/\Tau\Nu}(\Tau\Nu/\Nu,\mathcal{C}) \to \mathrm{Hom}_{\Gamma/\Tau\Nu}(\Tau\Nu/\Nu,\mathcal{K}) \to H^1(\Gamma/\Tau\Nu, \mathrm{Hom}(\Tau\Nu/\Nu,Z(\Nu))).$$
The group $H^1(\Gamma/\Tau\Nu, \mathrm{Hom}(\Tau\Nu/\Nu,Z(\Nu)))$ is finite, since $\mathrm{Hom}(\Tau\Nu/\Nu,Z(\Nu))$ is a free abelian group of finite rank. 
Hence the cokernel of $\mathrm{Hom}_{\Gamma/\Tau\Nu}(\Tau\Nu/\Nu,\mathcal{C})) \to \mathrm{Hom}_{\Gamma/\Tau\Nu}(\Tau\Nu/\Nu,\mathcal{K}))$ is finite. 
Therefore the cokernel of $\kappa_\ast: H^1(\Tau\Nu/\Nu, \mathcal{C})^{\Gamma/\Tau\Nu} \to H^1(\Tau\Nu/\Nu, \mathcal{K})^{\Gamma/\Tau\Nu}$ is finite. 
\end{proof}

\begin{lemma}  
The cokernel of $\kappa_\ast: H^1(\Gamma/\Nu, \mathcal{C}) \to H^1(\Gamma/\Nu, \mathcal{K})$ is finite. 
\end{lemma}
\begin{proof}
We have a short exact sequence $1 \to \Tau\Nu/\Nu \to \Gamma/\Nu \to \Gamma/\Tau\Nu \to 1$, and so 
we have a commutative diagram with horizontal exact sequences (cf.\,p.\,354 of \cite{M})
$$\begin{array}{ccccccc}
0= H^1(\Gamma/\Tau\Nu,\mathcal{C}) &\hspace{-.1in} \to \hspace{-.1in} & H^1(\Gamma/\Nu,\mathcal{C}) & \hspace{-.1in}\to \hspace{-.1in}& H^1(\Tau\Nu/\Nu,\mathcal{C})^{\Gamma/\Tau\Nu} & \hspace{-.1in}\to \hspace{-.1in} & H^2(\Gamma/\Tau\Nu,\mathcal{C}) = 0 
\vspace{.05in} \\  
\downarrow \alpha &  & \downarrow \beta  & & \downarrow \gamma & & \downarrow  \vspace{.05in} \\
0\to H^1(\Gamma/\Tau\Nu,\mathcal{K}) & \hspace{-.1in}\to \hspace{-.1in}& H^1(\Gamma/\Nu,\mathcal{K}) & \hspace{-.1in}\to \hspace{-.1in}& H^1(\Tau\Nu/\Nu,\mathcal{K})^{\Gamma/\Tau\Nu} &\hspace{-.1in} \to \hspace{-.1in} & H^2(\Gamma/\Tau\Nu,\mathcal{K}) \phantom{= 0}
\end{array}$$
where the homomorphism $\alpha, \beta, \gamma$ are induced by $\kappa : \mathcal{C} \to \mathcal{K}$. 
By the snake lemma (Lemma III.5.1 of \cite{H-S}), we have an exact sequence
$$\mathrm{coker}(\alpha) \to \mathrm{coker}(\beta) \to \mathrm{coker(\gamma}).$$
Now $\mathrm{coker}(\alpha) = H^1(\Gamma/\Tau\Nu,\mathcal{K})$ is finite by Lemma 8, and $\mathrm{coker(\gamma})$ is finite by Lemma 9. 
Hence $\mathrm{coker(\beta})$ is finite.  
Thus the cokernel of $\kappa_\ast: H^1(\Gamma/\Nu, \mathcal{C}) \to H^1(\Gamma/\Nu, \mathcal{K})$ is finite. 
\end{proof}

\begin{theorem} 
For each dimension $n$, there are only finitely many isomorphism classes of pairs of groups $(\Gamma, \Nu)$ such that $\Gamma$ 
is an $n$-space group and $\Nu$ is a normal subgroup of $\Gamma$ such that $\Gamma/\Nu$ is a space group. 
\end{theorem}
\begin{proof}
Let $m$ be a positive integer less than $n$. 
Let $\Mu$ be an $m$-space group and 
let $\Delta$ be an $(n-m)$-space group.  
Let $\mathrm{Iso}(\Delta,\Mu)$ be the set of isomorphism classes 
of pairs $(\Gamma, \Nu)$ 
where $\Nu$ is a normal subgroup of an $n$-space group $\Gamma$ 
such that $\Nu$ is isomorphic to $\Mu$ and $\Gamma/\Nu$ is isomorphic to $\Delta$. 
As there are only finitely many isomorphism classes of the groups $\Delta$ and $\Mu$ by 
Bieberbach's theorem \cite{B}, it suffices to prove that $\mathrm{Iso}(\Delta,\Mu)$ is finite. 

In \S 4, we defined a function $\omega: \mathrm{Iso}(\Delta,\Mu) \to \mathrm{Out}(\Delta,\mathrm{M})$ 
with $\mathrm{Out}(\Delta,\mathrm{M})$ finite by Lemma 4. 
The fibers of $\omega$ are finite by Lemmas 7 and 10. 
Therefore $\mathrm{Iso}(\Delta,\Mu)$ is finite. 
\end{proof}

In view of Theorem 10 of \cite{R-T}, Theorem 4 is equivalent to the following theorem. 

\begin{theorem} 
For each dimension $n$, there are only finitely many affine equivalence classes of geometric orbifold fibrations of compact, connected, flat $n$-orbifolds. 
\end{theorem}


\end{document}